\documentclass{article}

\usepackage{arxiv}

\usepackage[utf8]{inputenc} 
\usepackage[T1]{fontenc}    
\usepackage{hyperref}       
\usepackage{url}            
\usepackage{booktabs}       
\usepackage{amsfonts}       
\usepackage{nicefrac}       
\usepackage{microtype}      
\usepackage{lipsum}		
\usepackage{graphicx}
\usepackage{natbib}
\usepackage{doi}

\usepackage{amsmath,amssymb,amsbsy,amsfonts,latexsym,amsopn,amstext,amsxtra,euscript,amscd,hyperref}
\usepackage{natbib}
\usepackage{mathptmx}
\numberwithin{equation}{section}
\usepackage{amsmath}
\usepackage{color}
\usepackage{url}
\usepackage{multirow}
\usepackage{graphicx}
\usepackage{times}
\usepackage{amsthm}
\newtheorem{thm}{Theorem}

\newtheorem{definition}{Definition}
\usepackage[english]{babel}
\title{A New Robust Network Slack Based Measure Model}

\author{ {Alka Arya} \\
	Operations Management \& Decision Sciences Area,\\
	Indian Institute of Management Kashipur, \\
	Kashipur, U.S. Nagar, \\
	Uttarakhand, 244713,  India.\\
	\texttt{alka1dma@gmail.com} \\
	\And
	\href{https://orcid.org/0000-0003-0553-8187}{\includegraphics[scale=0.06]{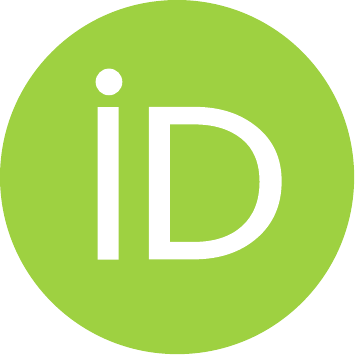}\hspace{1mm} Shubham Singh} \\
	Department of Mathematics,\\
	Indian Institute of Technology Kharagpur,\\
	Kharagpur, 721302, India. \\
	\texttt{shubham94iitkgp@gmail.com} \\
}

\renewcommand{\shorttitle}{A New Robust Network Slack Based Measure Model}

\begin{document}
\maketitle

\begin{abstract}
	In real-life challenges, unforeseen and unknown occurrences commonly influence the data values, which may affect the performance of the problems.  The performance of decision-making units (DMUs) is determined using the slack-based measure (SBM) model, which considers only crisp data values without uncertainty and is a black-box model. Many authors have used fuzzy SBM and stochastic SBM to deal with ambiguity and uncertainty, and many have used these approaches to deal with ambiguous and uncertain data. However, some ambiguous and uncertain data can not be taken as fuzzy logic and stochastic data due to the huge set of rules which must be given to construct a conceptual method. So, this paper tackles the uncertain data using robust optimization in which the input-output data are limited to remain within an uncertainty set, with extra constraints depending on the worst-case output for the uncertainty set and black-box is tackle with network SBM model. Thus, this paper proposes a robust network SBM model with negative, missing, and undesirable data that aims to maximize the efficiency same as traditional network SBM, and the constraint for a worst-case efficiency calculated by the uncertainty set its sustaining constraint while determining the efficiencies. Finally, the performances of Indian banks are computed, and then the results of the proposed models are compared to the crisp network SBM models and demonstrating how decision-makers can magnify the bank's performance of Indian banks.
\end{abstract}


Data uncertainty is unavoidable since the values are frequently influenced by unforeseen and unknown occurrences. Exogenous and endogenous data are two types of uncertain data \cite{lappas2018robust}. Exogenous uncertainty is rarely influenced by the decision-maker, whereas endogenous uncertainty is influenced by the decision maker's decisions, strategies, and judgments in a variety of application situations (see \cite{peeta2010pre}, \cite{edirisinghe2010optimized}, \cite{hassanzadeh2014robust}, \cite{hatami2021robustness}).  \cite{dantzig1955linear} proposed probabilistic constraint optimization, and  \cite{charnes1959chance} proposed stochastic programming, which may be traced back to the mid-1950s. These methods are based on the premise that the probability distributions of a random variable are perfectly defined. Furthermore, due to a lack of evidence in various real-world circumstances, the assumption is correct is unsure by the decision maker. Thereby, the literature has commonly been raised for the use of various ways to deal with uncertainty in quickly changing settings where stochastic models are ineffective. In terms of mathematics, uncertain optimization issues frequently confront two difficulties in practice: (i) an approach may not be feasible, and (ii) if it is, it may lead to a complete revenue (cost) that is significantly lower (higher) than the actual optimal value \cite{bertsimas2006robust}.  \cite{soyster1973convex} suggested a robust optimization framework for linear programming models in the early 1970s. In the context of unknown data, robust optimization seeks to develop a solution that is practical for all conceivable values of the parameters for achieving in the worst-case scenario \cite{hatami2021robustness}. Many real-life challenges require both immediate and long-term judgments. As a result, here-and-now decision variables are those that must be recognized right immediately, with no time to wait for more or complete information on uncertain parameters, and wait-and-see variables are those that can be recognized later when more information on uncertain parameters has been obtained \cite{yanikouglu2019survey}.

	The financial stability and growth of a country are dependent on finance. India is the world's second-largest country by population and geographically the world's seventh-largest territory. As a result of India's difficult economic situation, this study will analyze the performance of Indian banks using negative, incomplete, and undesirable data. Banks play a crucial role in growing India's economy in both urban and rural areas. Thereby, policymakers and planners are increasingly worried about the performance of the Indian banks. The operations of banks have undergone drastic changes. Economic growth occurs due to the liberalization of financial savings, their beneficial usage, and the transformation of many risks. Indian banks have seen significant changes in their operations due to privatization, globalization, and liberalization. In India, public sector banks confront a host of challenges, including high pricing and a lack of competition from private sector banks. In this technology era, the competition between the public sector (Government of India) and private sector banks is severe (governed by individuals). As a consequence, this study will calculate the performance of Indian banks.
	
A frontier approach to estimating the relative performance of decision-making units (DMUs) is Data Envelopment Analysis (DEA). The distance from the piece-wise linear frontier is used to calculate the performance of DMUs that generate multiple outputs by consuming multiple inputs in DEA. The DEA models were first suggested by  \cite{charnes1978measuring} to assess the relative efficiencies of DMUs. Constant returns to scale (CRS), indicated by  \cite{charnes1978measuring}, ignores the slacks when determining the efficiencies of DMUs.  \cite{tone2001slacks} suggested the input-oriented SBM and output-oriented SBM models to evaluate the performances of DMUs. The SBM model \cite{tone2001slacks} specifies the efficiencies and slacks of DMUs. As a result, the SBM model is more practical than the CRS. Due to the black-box nature of the SBM model, we cannot determine the internal structure of DMUs.

The black-box models were first proposed to assess the performance of DMUs using an overall system without using internal operations that used initial inputs to generate final DMU outputs, and this overall system was dubbed "black-box." We can't find the internal structure of DMUs using these black-box models. In certain real-world situations, such as banks, insurance firms, and businesses activities have two processes rather than one (see \cite{tavana2018new}, \cite{arya2021development}, \cite{ma2018evaluating}). We cannot determine the correct details about efficiencies using a black-box model, and some DMUs can be inefficient. A bank's overall efficiency is divided into two parts: operations and profitability efficiency.

According to \cite{wang1997use}, banks and industries have a two-stage profit process, with the capital collection being the first stage and investment being the second. This is why it's essential to look into the component or stage-by-stage process for determining DMU performance \cite{tavana2018new}. So, the explanation for any DMU's inefficiency can be determined. As a result, rather than using a black-box approach to evaluate DMU performance, the internal operation process (network) is more practical. The literature review on network DEA till 2014 was summarized by  \cite{kao2014network}. \cite{fare2000network} suggested a network DEA model with intermediate items and a heuristic approach for budget allocation.  \cite{maghbouli2014two} determined the efficiency of a Spanish airport Using a two-stage DEA model with undesirable data.  \cite{liu2015two} proposed a two-stage model with undesirable inputs, intermediates, and outputs for evaluating the performances and production possibility set (PPS) of China's Bank using the free disposal axiom.  \cite{mahdiloo2018modelling} calculated the technical, ecological, and process environmental efficiency of China's provinces using a proposed multi-objective two-stage DEA model.  \cite{mavi2019joint} suggested a two-stage DEA model focused on target programming to assess eco-efficiency and eco-innovation when dealing with undesirable and big data.  \cite{tajbakhsh2018evaluating} suggested a two-stage DEA model assess the output of fossil power plants. \cite{sotiros2019dominance} calculated the divisional additive efficiencies of DMUs using a network (mainly two-stage) method in DEA.

Some researchers consider deposits to be inputs (see \cite{das2009financial}, \cite{karimzadeh2012efficiency}) in banks, while others consider them to be outputs (see \cite{dash2009study}, \cite{laplante2015evaluation}). Deposits are regarded as intermediaries that play a dual role in determining a bank's results. Negative data is described as a loss, and some banks have missing data. For more details, researchers can visit the article of Kuosmanen \cite{kuosmanen2009data}. The existence of negative, missing, and undesirable data adds to the challenge. Despite the vast number of studies investigating the efficacy of two-stage systems, two-stage DEA models still consider non-negative data. There is no in-depth discussion of two-stage DEA models when negative data is present. In this study, the value for missing data is determined by using the imputation approach. Although negative values for inputs and outputs are becoming more common, negative values for intermediate variables are also permissible \cite{tavana2018new}. Negative data (like net profit, operating income) in DEA has also overgrown to assess the efficiencies of banks (see \cite{tavana2018new}, \cite{matin2011two}). In specific real-world applications, the outputs are in two forms: desirable and undesirable (see \cite{tavana2018new}). Desirable outputs are considered profitable/good for any DMU, whereas undesirable outputs are considered bad. As a result, undesirable outputs are viewed as inputs that must be reduced as much as possible during the manufacturing process. Since the idea of undesirable outputs was implemented, DEA has become a useful tool for assessing Bank's efficiency; Non-performing assets (NPAs) are regarded as undesirable outputs that must be reduced in tandem with inputs (see \cite{puri2014fuzzy}, \cite{nasseri2018fuzzy}) in DEA. NPAs avoid producing revenue in banks, and the bank's profit suffers as a result of their existence \cite{puri2017new}. Further, some inputs, intermediates, and outputs have no data value for a year due to various circumstances.

This paper also investigates when the input and output data of the SBM model are not known precisely.  In a real-world optimization problem, the most common reason for being uncertain data of the problems are given as follows: (i) Some parameters (returns, future needs, etc.) are not known precisely when the issue is addressed and are thus replaced by predictions. Therefore, these parameter entries are thus subject to prediction errors, (ii) Some of the parameters (contents associated with raw materials, parameters of the technological process) are difficult to quantify precisely. In reality, their values either follow some random variables or drift around the measured nominal values. These data are subject to measurement errors, (iii) Some of the information can not be implemented precisely as it is computed. The implementation mistakes that arise are the same as appropriate artificial data uncertainties. To describe imprecise data in DEA, three distinct types of methodologies have been proposed in the literature: fuzzy approaches, stochastic and robust optimization framework.

The uncertainty associated with the available input-output data is incorporated in the fuzzy DEA method by representing inputs and outputs as a membership function. However, solving fuzzy DEA with various membership functions is often quite tricky from a computational and practical point (see \cite{emrouznejad2014state}). If information of the input-output data follows some random variable with given probability density functions, then such problems undergo the stochastic DEA framework.  The chance-constrained programming technique is one of the most efficient and practical approaches to dealing with stochastic programming problems with probabilistic constraints that was initially proposed by \cite{charnes1959chance}. The most significant disadvantage of stochastic optimization is determining a random variable for the decision-maker is quite difficult  (see \cite{nasseri2018fuzzy}). For a comprehensive overview of stochastic optimization, one may refer to  \cite{ruszczynski2003stochastic},   \cite{birge2011introduction}, and  \cite{prekopa2013stochastic}. On the contrary, the Robust Optimization (RO) technique deals with uncertain optimization problems without making assumptions about the probability distribution functions. It is the most effective approach. The uncertain models take on an adversarial aspect, with the selected uncertain parameters from an uncertainty set having the worst impact on the model. Following that,  \cite{ben1998robust} studied robust convex optimization problems by taking into account the knowledge of uncertain data in the form of an ellipsoidal uncertainty set. Afterward, they approximate few significant general convex optimization problems, including linear programming, quadratically constrained programming, semi-definite programming, into the tractable formulations and solve them with an efficient technology such as polynomial-time interior-point methods. \cite{el1998robust} investigated robust semi-definite programming problems and established sufficient conditions to ensure that the robust solution is unique and continuous with respect to the unperturbed problem's data.  \cite{bertsimas2004price} studied robust linear optimization model with a budget uncertainty set and also discussed how to adjust the level of conservatism of the robust solutions in terms of probabilistic bounds of constraint violations.
Further, they extend their methods to solve discrete optimization problems. Further, \cite{ben1999robust} examined the robust counterpart of the linear programming problems with an ellipsoidal uncertainty set is computationally tractable, yielding a conic quadratic programming problem that can be solved in polynomial time. \cite{bertsimas2004robust} studied robust modeling of linear programming problems based on uncertainty sets given by the general norm.

The output of DMUs with uncertain data is determined using a robust optimization process. The robust optimization approach is used to solve a linear optimization problem with uncertain data. In this case, we use a robust optimization technique for creating the equivalents of robust of the two-stage SBM models to deal with extreme uncertainty in the data and make in-the-moment judgments. This method emphasizes tractability throughout the need to optimize this worst-case criterion while maintaining high degrees of robustness and objective values of high quality. This research proposes a set of static robust network SBM models for obtaining efficiencies in the presence of pricing uncertainty.   \cite{sadjadi2008data} proposed robust optimization in DEA and calculated the efficiencies of 38 Iranian electricity distribution firms, which was the first analysis of its kind in DEA.  \cite{sadjadi2011robust} proposed the super-efficiency DEA model and calculated the super-efficiencies of Iran's gas companies using robust optimization with an ellipsoidal uncertainty set.  \cite{omrani2013common} suggested a robust optimization in common weights in DEA and calculated the outputs of Iran's provincial gas companies.  \cite{lu2015robust} proposed a robust DEA model for determining DMU performance using a heuristic approach. \cite{toloo2019robust} presented robust and reduced robust CCR and BCC models with non-negative variables to predict the performance of 250 European banks. Recently, \cite{dehnokhalaji2021box} studied the DEA model, in which they considered that the information of input and output data is given in the form of box uncertainty set. A robust optimization approach is proposed to examine the performance measurement and ranking of DMUs with interval data. Therefore, a robust two-stage SBM model with negative, missing, and undesirable data could be proposed. The literature summary is provided in Table \ref{tab:supplyapp}.

%


	\begin{table}[!htb]
	\centering
	\caption{\textbf{Robust optimization in DEA}}\label{tab:supplyapp}
	\resizebox{\textwidth}{!}{
		\begin{tabular}{lllll}
			\hline
			Method  & Robust Approach & Application & Reference& Year \\\hline	
			BCC  & Ben-Tal and Nemirovski&  Iranian Electricity companies & \cite{sadjadi2008data} & 2008\\
			Interval CCR	 & Bertsimas and Sim & Empirical examples  & \cite{shokouhi2010robust}& 2010\\
			Super-efficiency CCR  & Bertsimas and Sim & Iranian Gas companies & \cite{sadjadi2011robust}& 2011\\
		CCR &  Bertsimas and Sim &  Iranian provincial gas companies & \cite{omrani2013common}& 2013\\
		Interval CCR &  Bertsimas and Sim &  Empirical examples & \cite{shokouhi2014consistent}& 2014\\
			BCC	& Ben-Tal and Nemirovski &   Vehicle routing problem & \cite{lu2015robust}& 2015\\
			Interval CCR	& Bertsimas and Sim &   Iranian banks & \cite{aghayi2016efficiency}& 2016\\
		Super-efficiency CCR	& Bertsimas and Sim &   Forest districts \& Iranian gas companies & \cite{arabmaldar2017new}& 2017\\
			BCC	& Bertsimas and Sim &   European banks & \cite{toloo2019robust}& 2019\\
		CCR	& Bertsimas and Sim &   Empirical examples & \cite{salahi2020new}& 2020\\
		CCR	& Bertsimas and Sim &   US manufacturing & \cite{hatami2021robustness}& 2021	\\
	CCR	& Ben-Tal and Nemirovski& East Virgina hospitals& \cite{dehnokhalaji2021box}&2021\\\hline
	\end{tabular}}
\end{table}	
The significant contributions of this paper are organized as follows: (1) this article investigates a two-stage SBM model in the uncertain environment using robust optimization (RO) theory. Additionally, it demonstrates how the model is used to calculate the stage efficiencies of homogeneous DMUs, as well as their overall efficiencies, (2) this article also determines the efficiencies of DMUs when the negative, missing, and undesirable data is presented, and a robust two-stage SBM model with negative, missing, and undesirable data is proposed, (3) three robust two-stage SBM models are offered with types of tractability in RO: ellipsoidal, polyhedral, budget uncertainty set, (4) this article analyses the relative efficiencies of Indian bank branches using robust two-stage SBM model to incorporate the proposed models.

The rest of the paper is laid out as follows. Section 2 provides some background information on basic SBM models, undesirable data negative data, missing data, two-stage SBM models, and robust optimization, which we use to propose the robust two-stage models. Section 3 proposes a robust two-stage SBM model to deal with missing, negative, and undesired data and examine the suggested models' features. We give a case study of Indian banks in Section 4 to show the applicability and efficacy of the proposed models and processes. In Section 5, we describe our findings.
	\section{Preliminaries}
	This section explains the basic SBM model, undesirable outputs, negative outputs, and robust optimization techniques.
	\subsection{\bf Basic SBM model} Assume that n homogeneous DMUs ($DMU_j; j=1,2,...,n$) are taken for measuring the performances. Let us suppose that each DMU utilizes m inputs to produce s outputs. Let $y_{rk};\,(r=1,2,...,s)$ be the $r^{th}$ output amount produced to utilize $x_{ik};\,(i=1,2,...,m)$ the $i^{th}$ input amount by $DMU_{k}$, shown in Figure \ref{fig:minputs}.
	\begin{figure}[!htb]
		\centering
		\includegraphics[width=0.45\linewidth]{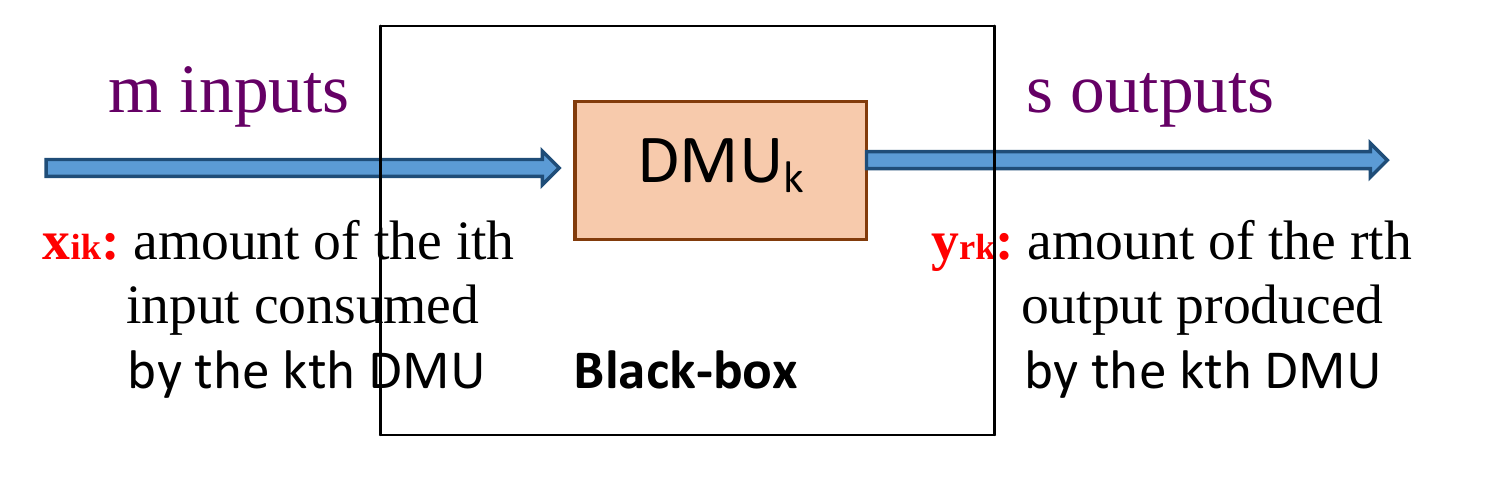}
		\caption{Graphical representation of black-box}
		\label{fig:minputs}
	\end{figure}
	
	The fractional SBM model \cite{tone2001slacks} for ${DMU_{k}}; \,\,k=1,2,...,n$ is given in Table \ref{tab:1},
	\begin{table}[!htb]
		\centering
		\caption{{SBM model}}
		\label{tab:1}
		\begin{tabular}{cc}
			\hline
			SBM: Fractional model & \,\,\,\,\,\,\,\,\,\,\,\,\,\,\,SBM: LP model\\\hline
			$\min P_{k}=\dfrac{1-\dfrac{1}{m}{\sum\limits_{i=1}^{m}{s_{ik}^-}/{x_{ik}}}}{1+\dfrac{1}{s}\sum\limits_{r=1}^{s}{s_{rk}^+}/{y_{rk}}}$ & \,\,\,\,\,\,\,\,\,\,\,\,\,\,\,$\max P_{k}=p+\dfrac{1}{s}{\sum\limits_{r=1}^{s}{S_{rk}^+}/{y_{rk}}}$\\
			subject to\,\, $ x_{ik}=\sum\limits_{j=1}^{n}{x_{ij}\lambda_j+s_{ik}^-}\,\, \forall i,$  & \,\,\,\,\,\,\,\,\,\,\,\,\,\,\,subject to $ p-\dfrac{1}{m}{\sum\limits_{i=1}^{m}{S_{ik}^-}/{x_{ik}}} =1,$ \\
			$\,\,\,\,\,\,\,\,\,\,\,\,y_{rk}=\sum\limits_{j=1}^{n}{y_{rj}\lambda_j-s_{rk}^+}\,\,\,\,\forall r,$ & \,\,\,\,\,\,\,\,\,\,\,\,\,\,\,$px_{ik}=\sum\limits_{j=1}^{n}{x_{ij}\lambda_j'+S_{ik}^-}\,\, \forall i,$ \\
			$\sum\limits_{j=1}^{n}{\lambda_j}=1,	$& \,\,\,\,\,\,\,\,\,\,\,\,\,\,\,$py_{rk}=\sum\limits_{j=1}^{n}{y_{rj}\lambda_j'-S_{rk}^+}\,\,\,\,\forall r,$\\
			$s_{ik}^-\geq 0\,\,\forall i,\,\,\,s_{rk}^+\geq0\,\,\forall r.$	&$\sum\limits_{j=1}^{n}{\lambda_j'}=p\, \, \forall j,\,\, p >0,$ \\
			& $S_{ik}^-=ps_{ik}^-\geq0\,\,\forall i,$\\
			& $S_{rk}^+=ps_{rk}^+\geq0\,\,\forall r,\,\,\lambda_j'={p}{\lambda_j}\geq 0\, \, \forall j$ .\\
			\\\hline
		\end{tabular}
	\end{table}
where $s^{-}_{ik}$ is the slack variable of $i^{th}$ input and $s^{+}_{rk}$ is the slack variable of the $r^{th}$ output of the $k^{th}$ DMU. Let $P_k^*$ be the optimal efficiency of SBM model for $DMU_k$. Then, $DMU_k$ is called SBM efficient if $P_k^*=1$. This condition is equivalent to $s^{-}_{ik}=0$ and $s^{+}_{rk}=0$, i.e., no output 	shortfalls and no input excesses in optimal solution, otherwise $DMU_k$ is SBM inefficient \cite{tone2001slacks}.
By normalizing the denominator, the fractional SBM model is turned into an LP SBM model. In a fractional SBM model with addition constraint, the denominator is maximized using the normalized approach while the numerator is equal to 1. The output-oriented SBM model is the name given to this LP SBM model. Researchers can read \cite{tone2001slacks} article  for further information about SBM.
\subsection{\bf Undesirable data}	
The premise in traditional DEA models is that inputs must be minimized and outputs must be maximized. It's worth mentioning, however, that processes can also produce undesirable consequences, such as Non-performing assets \cite{puri2014fuzzy}. Other applications, such as health care, banks, and business, may potentially produce undesirable outputs (see \cite{scheel2001undesirable}, \cite{lu2012data}, \cite{puri2014fuzzy}, \cite{tone2004dealing}). A DMU's efficiency is defined by a production process that takes $m$ inputs and produces $S$ outputs, with $s_1$ outputs being desirable (good) and $s_2$ outputs being undesirable (poor), with $s=s_1+s_2$. Let $x_{ij}(\forall i)$ represent the $m$ inputs to $DMU_k$, and $y_{r_1k}^D\,\, (r_1=1,2,...,s_1)$ and $y_{r_2k}^D\,\, (r_2=1,2,...,s_2)$ represent the $s_1$ desirable and $s_2$ undesirable outputs produced by $DMU_k$, respectively. The following is the SBM model with undesirable outputs (\cite{tone2004dealing}, \cite{apergis2015energy}):
\begin{small}
\begin{flalign}
&\min \rho_{k}=\dfrac{1-\dfrac{1}{m}{\sum\limits_{i=1}^{m}{s_{ik}^-}/{x_{ik}}}}{1+\dfrac{1}{s_1+s_2}(\sum\limits_{r_1=1}^{s_1}{s_{r_1k}^D}/{y_{r_1k}^D}+\sum\limits_{r_2=1}^{s_2}{s_{r_2k}^U}/{y_{r_2k}^U})} \nonumber&\\
&\mbox{ Subject to}\nonumber&\\
&{x_{ik}=\sum\limits_{j=1}^{n}{x_{ij}\lambda_j+s_{ik}^-}\,\, \forall i, \,\,\,
y_{r_1k}^D=\sum\limits_{j=1}^{n}{y_{r_1j}\lambda_j-s_{r_1k}^D}\,\,\,\,\forall r_1,\,\,\,y_{r_2k}^U=\sum\limits_{j=1}^{n}{y_{r_2j}\lambda_j+s_{r_2k}^U}\,\,\,\,\forall r_2,\,\,\,s_{ik}^-\geq 0\,\,\forall i,\,\,\,s_{r_1k}^D\geq0,\,\,s_{r_2k}^U\geq0\forall r.}\nonumber&
\end{flalign}
\end{small}
The $DMU_k$ is efficient when $\rho_{k}=1,\,s_{ik}^-=0,\, s_{r_1k}^D=0,\, s_{r_2k}^U=0$ \cite{tone2004dealing}. Researchers can read Tone's \cite{tone2004dealing} and Yang et al.'s article \cite{yang2018regional} for further information about undesirable SBM.
\subsection {\bf Negative data}
We employ the methodology proposed by  \cite{portela2004negative} to avoid issues arising from the existence of negative data while ensuring the usage of a meaningful directional vector. Their method is based on a modified version of the generic directional distance model, in which the direction chosen determines whether improvements in both inputs and outputs are sought, resulting in either the easiest or most difficult targets to achieve (\cite{portela2004negative}, \cite{tavana2018new}). Their method embraces negative data without needing any transformations and produces an efficiency measure that is similar to radially evaluations in classical DEA employing the concept of range of feasible improvements (\cite{portela2004negative}, \cite{tavana2018new}). \cite{portela2004negative} introduces the technical features of negative data, with an ideal point $I=(\max\limits_j{y_{1j}},....,\max\limits_j{y_{sj}},\min\limits_j{x_{1j}},...,\min\limits_j{x_{mj}},)$ and defined the directional vector $(R_j^x,R_j^y)=(R_{1j}^x,..., R_{mj}^x, R_{1j}^y,...,R_{sj}^y)$ as the direction from $DMU_{j_0}$ to the ideal point $I$, that is: $R_{ip}^x=x_{ip}-\min\limits_{j}{x_{ij}}$ and $R_{rp}^y=\max\limits_{j}{y_{rj}}-y_{rp}$. Researchers can read \cite{portela2004negative} article  for further information about negative data.

\subsection {\bf Missing data}
In many cases, coverage of potentially relevant input-output variables is poor, or enterprises fail to disclose all required information. In DEA applications, the problem of missing observations is a chronic sickness that affects both the quality and quantity of data \cite{kao2000data}. The immediate cure, of course, is to devote more time and effort to data collection. On the other hand, DEA applications often rely on non-experimental, observational data, as opposed to data gathered in laboratory experiments or field trials in the natural sciences. The majority of our work is on "found" data, material that has been acquired by someone else, frequently for entirely different goals \cite{kuosmanen2009data}. We almost always have to accept missing data as an unfortunate part of real-world data in these situations. Unbalanced data with blank entries are increasingly commonly used in panel data regression models. The treatment of missing data, on the other hand, has received only a few fleeting comments in the DEA literature \cite{kuosmanen2009data}. However, missing data are just as common in DEA as they are in regression analysis.  As a result, the conventional way to solving this problem is to remove blank entries from data matrices at random. Some writers openly disclose the exclusion of firms and/or variables from the data set. Still, it's reasonable to assume that many more do so in the pre-processing phase of the data analysis without disclosing the selection rationale. The issue is that there are often many different approaches to build balanced data, and the decision of which firms and factors to include in the data set can have a significant impact on the outcomes (\cite{kao2000data}, \cite{kuosmanen2009data}). The regression analysis approach, mice package in R, is used in this paper to transform missing data. More details on the mice package can be found in \cite{van2011mice}.

\subsection{\bf Two-stage SBM}	
Let us assume that the operation of a $DMU_j, \,j=1 \,\hbox{to}\, n$ with two-stages (systems). Suppose that the efficiency of first stage (system) is represented by $E_k^1$ which utilize the $m$ inputs, $x_{ij}$, to produce $d$ outputs (intermediates), $z_{dj}$, and the efficiency of second stage (system) is represented by $E_k^2$ which utilize $d$ inputs (intermediates) to produce $r$ outputs. Let $DMU_k,\,k=a \,\hbox{ to }\,n$ is the targeted DMU which efficiency is to be determined. Both the systems (stages) are connected in series. Suppose the input and output (intermediate) weights for first stage are $S_{ik}^-$ and $S_{dk}$ respectively, and the input (intermediate) and output weights for second stage are $S_{dk}$ and $S_{rk}^+$ respectively. Intermediate works in two ways, in first stage it is taken as outputs and in second stage it is taken as inputs. The two-stage SBM model (\cite{tone2004dealing}, \cite{apergis2015energy}) is as follows:

	\begin{minipage}[b]{0.4\textwidth}
	\begin{flalign}\label{mod001}
		&\max F_{k}^1= p_1+{\dfrac{1}{D}\left({\sum\limits_{d=1}^{D}{{S_{dk}}}/{z_{dk}}}\right)}\nonumber&\\
		&\mbox{subject to}&\nonumber\\
		&p_1-\dfrac{1}{m}{\sum\limits_{i=1}^{m}{{S_{ik}^-}}/{x_{ik}}}=1,\nonumber&\\
		&p_1 x_{ik}=\sum\limits_{j=1}^{n}{x_{ij}\lambda_j'}+{{S_{ik}^-}},\,\,\forall i,\nonumber&\\
		& p_1 z_{dk} =\sum\limits_{j=1}^{n}{z_{dj}\lambda_j'}-{{S_{dk}}},\,\,\,\,\forall d,\,\,\,\,\,\nonumber&\\
		& \sum\limits_{j=1}^{n}{\lambda_j'}=p_1,\,\,p_1 >0;\,\,\lambda_j'\geq 0, \,\forall j,\nonumber&\\
		&{S_{ik}^-}\geq 0,\,\,\forall i,\,\,\,{S_{dk}},\,\,\forall d\geq 0.&
		\end{flalign}
		\end{minipage}
		\hfill
		\begin{minipage}[b]{0.48\textwidth}
			\begin{flalign}\label{mod002}
	&\max F_{k}^2= p_1+\dfrac{1}{s}{\left(\sum\limits_{r=1}^{s}{{S_{rk}^+}}/{y_{rk}}\right)} \nonumber&\\
	&\mbox{subject to}&\nonumber\\
	&p_1-{\dfrac{1}{D}\left({\sum\limits_{d=1}^{D}{{S_{dk}}}/{z_{dk}}}\right)}=1,\nonumber&\\
	& p_1 z_{dk} =\sum\limits_{j=1}^{n}{z_{dj}\lambda_j'}+{{S_{dk}}},\,\,\,\,\forall d,\nonumber&\\
	& p_1 y_{rk} =\sum\limits_{j=1}^{n}{y_{rj} \lambda_j'}-{{S_{rk}^+}},\,\,\,\,\forall r_1,\nonumber&\\
	& \sum\limits_{j=1}^{n}{\lambda_j'}=p_1,\,\,p_1 >0;\,\,\lambda_j'\geq 0, \,\forall j,\nonumber&\\
	& {S_{rk}^+}\geq 0,\,\,\forall r\,\,\,{S_{dk}}\geq 0,\,\,\forall d.&
	\end{flalign}
	\end{minipage}\\

The notations are same as previously explained. Let $E_k^1=1/{F_k^1}$, and $E_k^2=1/{F_k^2}$. $DMU_k$ is called SBM efficient in two-stage method if $E_k^{1*}=1$ and $E_k^{2*}=1$, otherwise, SBM inefficient, where $(.)^{*}$ is the optimal value of $(.)$.
\subsection{\bf Robust optimization}
	In this section, we give the first brief introduction about the robust optimization problem, a class of convex optimization problems in which we take the somewhat agnostic view that our problem data is not accurate. In robust optimization, we are willing to accept a sub-optimal solution for the nominal data values to ensure that the solution remains feasible and close to optimal solutions when the data changes. A concern with such an approach is that it might be too conservative.
	
The general mathematical formulation of a linear optimization problem is as follows:

	\begin{eqnarray}\label{ar1}
	  &&  {\mbox{min}} ~~\sum_{j}^{n} c_{j} x_{j}\\
	&& {\hbox{subject to  }} \sum_{j}^{n} a_{ij} x_{j} \le b_{i}, ~~\hbox{where}~~ (a_{ij}, b_{i} )\in \mathcal{U}_{i}~~\forall~i,\\\label{ar2},
	&& x_{j} \ge 0, ~~\forall~j,
	\end{eqnarray}
where $c\in \mathcal{R}^{n}$ is cost parameter, $x \in \mathcal{R}^{n}$ is decision variables, $a_{ij}\in \mathcal{R}^{m \times n}$ is coefficient matrix, $ b\in\mathcal{R}^{m}$ and $\mathcal{U}_{i} \in \mathcal{R}^{L}$ is uncertainty sets, which, for our purposes, will always be convex and compact. In general, the robust model can not be solved unambiguously, because many classes of models admit efficient solutions. In the above linear optimization model \eqref{ar1}-\eqref{ar2}, we consider that the cost parameter $ c$ is fixed. Then, the feasible set of the problem \eqref{ar1}-\eqref{ar2} is defined as:$
    X(\mathcal{U})= \left\{x \Big|  \sum_{j}^{n} a_{ij} x_{j} \le b_{i},  (a_{ij}, b_{i} )\in \mathcal{U}_{i}~~\forall~i. \right\}
$\\
Here, a constraint-wise uncertainty set has been considered. In order to approximate the set of constraints into the tractable formulation, first, we transform it into the robust counterpart, which is given as follows:
\begin{eqnarray}\label{ar3}
	&& \sum_{j}^{n} a_{ij}(\xi) x_{j} \le b_{i}(\xi), ~~\hbox{where}~~ \xi \in \mathcal{Z}_{i},~~\forall~i,
\end{eqnarray}
where $\xi \in \mathcal{R}^{L}$ is possible deviations and $\mathcal{Z} \subset \mathcal{U}$ is primitive uncertainty set. Now, by defining the affine function of uncertain parameter $a_{ij}(\xi)=a_{ij}^{0}+\sum_{\ell=1}^{L} \xi_{\ell} a_{ij}^{\ell}$ and $b_{i}(\xi)=b_{i}^{0}+\sum_{\ell}^{L} \xi_{\ell} b_{i}^{\ell}$ where $a_{ij}^{0}, b_{i}^{0}$ are the nominal value and $a_{ij}^{\ell}, b_{i}^{\ell}$ are the possible deviations, to transform into tractable formulation of \eqref{ar3} for any possible realization of $\xi$ in the given uncertainty set $\mathcal{U}$ solution remains feasible, Mathematically, it can be expressed as:
\begin{equation}\label{ar4}
    \sum_{j}^{n} a_{ij}^{0} x_{j}+ \max_{\xi \in \mathcal{Z}_{i}} \sum_{j}^{n} \sum_{\ell=1}^{L} \xi_{\ell}\left( a_{ij}^{\ell} x_{j}- b_{i}^{\ell}\right) \le b_{i}^{0},
\end{equation}

The equation \eqref{ar4} known as sub-problem of equation \eqref{ar3}, that must be solved. The complexity of addressing the Robust Optimization issue is determined by its structure. Here, three different types of uncertainty sets have been considered with some parameters to control the size of the uncertainty set.
The structure of constructing these uncertainty sets is discussed in \cite{bertsimas2011theory}.

\subsubsection*{\textbf{Ellipsoidal uncertainty set}}
Ellipsoidal uncertainty sets have been referred to by  \cite{ben1999robust} as well as  \cite{el1997robust, el1998robust}. Controlling the size of these ellipsoidal sets interprets a budget of uncertainty that the decision-maker selects to conveniently trade-off between robustness and performance. Mathematically, it is represented as:
\begin{equation*}
    \mathcal{Z}_{i}=\{\xi|~~\|\xi\|_{2}\le \Omega_{i}, ~\forall~i,\}
\end{equation*}
where $\Omega_{i} \in [0, \sqrt{L}]$ is adjustable robust parameter and $L$ is the number of possible deviations. In this set, all the uncertain parameters within the radius $\Omega$ from the nominal values of the parameter are considered.


\subsubsection*{\textbf{Polyhedral uncertainty set}}
Polyhedral uncertainty is a particular case of ellipsoidal uncertainty that arises from the intersection of a finite number of ellipsoidal cylinders that is described in \cite{ben1999robust}. Mathematically, the polyhedral uncertainty set is defined as:
\begin{equation*}
    \mathcal{Z}_{i}=\{\xi|~~D_{i}^{T}\xi+q_{i} \ge 0 ~\forall~i,\}
\end{equation*}
 where $D_{i}\in \mathcal{R}^{L \times K}$ is deviations matrix of each constraint and $q_{i}\in\mathcal{R}^{K}$.



\subsubsection*{\textbf{Budget uncertainty set}}
\cite{bertsimas2004price} exploit this duality with a family of polyhedral sets that express a budget of uncertainty in terms of cardinality constraints, which is opposed to the size of an ellipsoid. In this set, they control the number of parameters that are allowed to deviate from their nominal values and provide a trade-off between the optimality and the robustness to parameter perturbation.

In budget uncertainty set, each element of coefficient matrix vary in some interval about their nominal value, i.e., $ a_{ij}\in [a_{ij}^{0}-\hat a_{ij},\,a_{ij}^{0}+\hat a_{ij}]$, where $a_{ij}^{0}$ is the nominal value and $\hat a_{ij}$ is half length of the confidence interval. In the original model of \cite{soyster1973convex}, rather than protect against the case when every parameter can deviate, we allow at most $\Gamma_{i} \in[0,~n]$ coefficients to deviate. The non-negative number $\Gamma_{i}$ represents the budget parameter in budget uncertainty set for every $i$-th constraint of the equation \eqref{ar4}.
If $\Gamma_{i}=0$, it means there is no coefficients parameter are deviate. In this case, equation \eqref{ar4} is similar to its nominal constraints, and if $\Gamma_{i}=n$, then all the coefficients parameters have deviated around their nominal values. In this case, uncertainty set $\mathcal{U} $ is defined by:
\begin{equation}
    \mathcal{U}_{i} = \left\{ a_{ij}\Big| a_{ij}=a_{ij}^{0}+\eta_{ij}\hat a_{ij}, \sum_{j}^{n} \frac{a_{ij}-a_{ij}^{0}}{\hat a_{ij}} \le \Gamma_{i}, ~\forall ~i\right\}
\end{equation}
where, $\eta_{ij}$ is the scaled deviation defined by $\eta_{ij}=(a_{ij}-a_{ij}^{0})/\hat a_{ij}$, which in support set $[-1,1] \subset \mathcal{Z}_{i}$. Therefore, $\mathcal{Z}_{i}$ becomes:
\begin{eqnarray}
\mathcal{Z}_{i}= \left\{ \eta_{ij} \Big| \sum_{j=1}^{n} \eta_{ij} \le \Gamma_{i}, ~-1\le \eta_{ij}\le 1,~~~\forall ~i,j\right\}.
\end{eqnarray}

	\section{Proposed robust network SBM model}	
Assume that the operation of a  $DMU_j\,(j=1,2,3,...,n)$ with two sub-systems, as can be seen in Figure \ref{fig:network2stage}. Assume that we want to evaluate the performance of a group of n homogeneous DMUs ($DMU_j,\,j=1,2,...,n$). Assume each DMU uses $m$ inputs to generate $D$ outputs in the first stage. Let $x_{ij};\,i = 1,2,...,m,$ be the amount of the $i$th input utilized and $z_{dj};\,\,d = 1,2,...,d,$ and $y_{rj}; \, r = 1,2,...,s$ be the amount of the $d$th and $r$th outputs respectively produced by $DMU_{j}$. In second stage, assume that each DMU utilizes $m_2$ and $d$ inputs to produce $S=s_1+s_2$ outputs. Let $x_{i_2j};\,i_2 = 1,2,...,m_2,$ and $z_{dj}$ be the amount of the $i_2$th and $d$th input utilized and $y_{r_1j}^D;\,r_1 = 1,2,...,s_1,$ and $y_{r_2j}^U;\,r_2 = 1,2,...,s_2,$ be the amount of the $s_1$th and $s_2$th desirable and undesirable final outputs produced respectively by $DMU_{j}$. $z_{dj}$ has dual role, it is output in first stage and input in second stage and known as intermediate. The first stage (Stage 1) is then linked to second stage (Stage 2) in series. This two-stage SBM model is also known as SBM series system.	The graphical representation is given in Figure \ref{fig:network2stage}.
	\begin{figure}[!htb]
		\centering
		\includegraphics[width=0.5\linewidth]{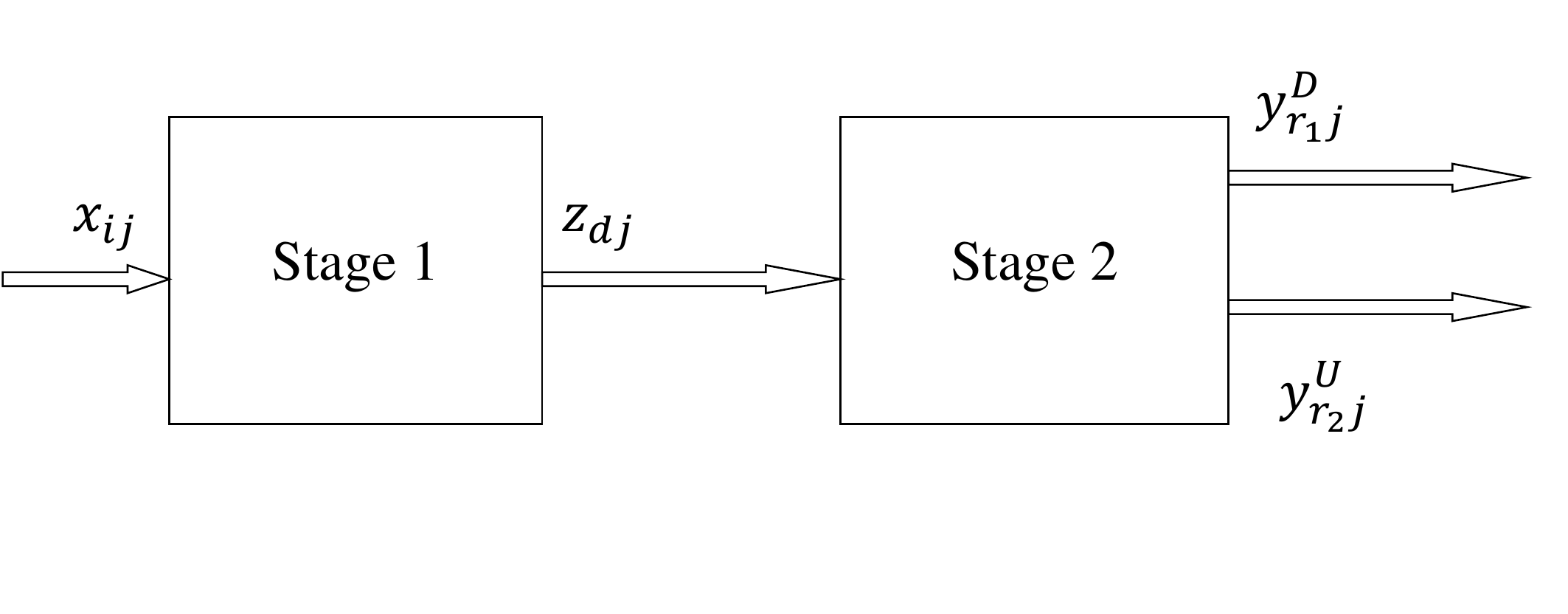}
		\caption{Graphical representation of two-stage}
		\label{fig:network2stage}
	\end{figure}
\\	The first stage LP SBM model for ${DMU_{k}},\,k=1,2,3,...,n$ is given in Model \ref{mod1},

	SBM model for Stage 1:
	\begin{flalign}\label{mod1}
	&\max F_{k}^1= p_1+{\dfrac{1}{D}\left({\sum\limits_{d=1}^{D}{{S_{dk}}}/{z_{dk}}}\right)}\nonumber&\\
	&\mbox{subject to } \Bigg\{p_1-\dfrac{1}{m}{\sum\limits_{i=1}^{m}{{S_{ik}^-}}/{x_{ik}}}=1,\,\,\,p_1 x_{ik}=\sum\limits_{j=1}^{n}{x_{ij}\lambda_j'}+{{S_{ik}^-}},\,\,\forall i,\,\,	p_1 z_{dk} =\sum\limits_{j=1}^{n}{z_{dj}\lambda_j'}-{{S_{dk}}},\,\,\,\,\forall d,\nonumber& \\
	& \sum\limits_{j=1}^{n}{\lambda_j'}=p_1,\,\,p_1 >0;\,\,\lambda_j'\geq 0, \,\forall j, \,\,\,{S_{ik}^-}\geq 0,\,\,\,{S_{dk}}\geq 0.\Bigg\}&
	\end{flalign}
	where $S^{-}_{ik}$, $S_{dk}$ and $S^{+}_{rk}$ are the slack variables corresponding to $i$th input, $d$th intermediate and $r^{th}$ output respectively for $DMU_k$.
	\begin{definition}
		Let $\theta_{k}^1=1/F_{k}^1$. $\theta_{k}^1 $ is called the first stage SBM (FSSBM) efficiency (FSSBME) of $DMU_{k}$. Let $\theta_{k}^{1*}$ be the optimal value of $\theta_{k}^1$. Then $DMU_{k}$ is called FSSBM-efficient if $\theta_{k}^{1*}  = 1$; otherwise, $DMU_{k}$ is called FSSBM-inefficient. Model \ref{mod1} is the proposed FSSBM (PFSSBM) model.
	\end{definition}
	The second stage LP SBM model for ${DMU_{k}},\,k=1,2,...,n$ is given in Model \ref{mod2}, where the output data is in desirable and undesirable form.
	\begin{flalign}\label{mod2}
	&\max F_{k}^2= p_1+\dfrac{1}{s_1+s_2}{\left(\sum\limits_{r_1=1}^{s_1}{{S_{r_1k}^+}}/{y_{r_1k}^D}-\sum\limits_{r_2=1}^{s_2}{{S_{r_2k}^+}}/{y_{r_2k}^U}\right)} \nonumber&\\
	&\mbox{subject to  }\Bigg\{p_1-{\dfrac{1}{D}\left({\sum\limits_{d=1}^{D}{{S_{dk}}}/{z_{dk}}}\right)}=1,\,\,\, p_1 z_{dk} =\sum\limits_{j=1}^{n}{z_{dj}\lambda_j'}+{{S_{dk}}},\,\,\,\,\forall d,\,\,\, p_1 y_{r_1k}^D =\sum\limits_{j=1}^{n}{y_{r_1j}^D \lambda_j'}-{{S_{r_1k}^+}},\,\,\,\,\forall r_1,\nonumber&\\
	& p_1 y_{r_2k}^U =\sum\limits_{j=1}^{n}{y_{r_2j}^U \lambda_j'}+{{S_{r_2k}^+}},\,\,\,\,\forall r_2,\,\,\, \sum\limits_{j=1}^{n}{\lambda_j'}=p_1,\,\,p_1 >0;\,\,\lambda_j'\geq 0, \,\forall j, \,\,\,{S_{r_1k}^+}\geq 0,\,\,\,{S_{r_2k}^+}\geq 0,\,\,\,{S_{dk}}\geq 0.\Bigg\}&
	\end{flalign}
	where $S^{-}_{i_2k}$, $S_{dk}$, $S^{+}_{r_1k}$ and $S^{+}_{r_2k}$ are the slack variables corresponding to $i_2$th input, $d$th intermediate, $r_1$th desirable output and $r_2$th undesirable output respectively for $DMU_k$.
	\begin{definition}
		Let $\theta_{k}^2=1/F_{k}^2$. $\theta_{k}^2 $ is called the Second stage SBM (SSSBM) efficiency (SSSBME) of $DMU_{k}$. Let $\theta_{k}^{2*}$ be the optimal value of $\theta_{k}^2$. Then $DMU_{k}$ is called SSSBM-efficient if $\theta_{k}^{2*}  = 1$; otherwise, $DMU_{k}$ is called SSSBM-inefficient. Model \ref{mod2} is the proposed SSSBM (PSSSBM) model.
	\end{definition}

	\subsection{Equality constraints in uncertainty}
	In robust optimization, if constraints have equal sign in any linear programming (LP) then the region may be feasible or infeasible (\cite{ben2009robust}, \cite{gorissen2015practical}). LP SBM model nclude the uncerainty in the normalization constraint when fractional SBM model is transformed into LP SBM model using normalized the denominator (The denominator of fractional SBM model is equal to 1). So, the uncertainty analysis in FSSBM and SSSBM models include the uncertainty in the normalization constraint (first constraint) in Models \ref{mod1} and \ref{mod2} respectively. If the data is uncertain in FSSBM and SSSBM models then the robust optimization is not suitable for Models \ref{mod1} and \ref{mod2}. The equality constraint in Model \ref{mod1} has uncertain values for robust optimization \cite{gorissen2015practical}.

	
	\begin{flalign}\label{mod3}
	&\max = p_1+{\dfrac{1}{D}\left({\sum\limits_{d=1}^{D}{{S_{dk}}}/{z_{dk}}}\right)}\nonumber&\\
	&\mbox{subject to  }\Bigg\{p_1-\dfrac{1}{m}{\sum\limits_{i=1}^{m}{{S_{ik}^-}}/{x_{ik}}}=t,\,\,\,p_1 x_{ik}=\sum\limits_{j=1}^{n}{x_{ij}\lambda_j'}+{{S_{ik}^-}},\,\,\forall i,\,\,\,p_1 z_{dk} =\sum\limits_{j=1}^{n}{z_{dj}\lambda_j'}-{{S_{dk}}},\,\,\,\,\forall d,\,\,\,\sum\limits_{j=1}^{n}{\lambda_j'}=p_1,\nonumber&\\
	& p_1 >0;\,\,\lambda_j'\geq 0, \,\forall j, \,\,\,{S_{ik}^-}\geq 0,\,\,\,{S_{dk}}\geq 0.\Bigg\}&
	\end{flalign}
$(S_{ik}^{-*},\,S_{dk},\,\lambda_j^{'*})$ is an optimal solution of Model \ref{mod1} if and only if $(t S_{ik}^{-*},\,t S_{dk},\,t \lambda_j^{'*})$ is an optimal solution of Model \ref{mod3}.

	\begin{thm}
		The equivalent model of Model \ref{mod1} is given below:
		\begin{flalign}\label{mod4}
		&\max = p_1+{\dfrac{1}{D}\left({\sum\limits_{d=1}^{D}{{S_{dk}}}/{z_{dk}}}\right)}\nonumber&\\
		&\mbox{subject to   } \Bigg\{p_1-\dfrac{1}{m}{\sum\limits_{i=1}^{m}{{S_{ik}^-}}/{x_{ik}}}\leq t,\,\,\,p_1 x_{ik}=\sum\limits_{j=1}^{n}{x_{ij}\lambda_j'}+{{S_{ik}^-}},\,\,\forall i,\,\,\, p_1 z_{dk} =\sum\limits_{j=1}^{n}{z_{dj}\lambda_j'}-{{S_{dk}}},\,\,\,\,\forall d,\,\,\, \sum\limits_{j=1}^{n}{\lambda_j'}=p_1,\,\,\nonumber&\\
		&p_1 >0;\,\,\lambda_j'\geq 0, \,\forall j, \,\,\,{S_{ik}^-}\geq 0,\,\,\,{S_{dk}}\geq 0.\Bigg\}&
		\end{flalign}
	\end{thm}
	\begin{proof}
		We have proved that Model \ref{mod1} is equivalent to Model \ref{mod3}. Now, we will show that Model \ref{mod3} is equivalent to Model \ref{mod4}. The dual of Model \ref{mod3} is given in \ref{mod5} and the dual of Model \ref{mod4} is given in \ref{mod6}.\\
	\begin{minipage}[b]{0.42\textwidth}
	\begin{flalign}\label{mod5}
		&\min = t \theta &\nonumber\\
		&\mbox{subject to}&\nonumber\\
		& \sum\limits_{i=1}^{m} u_{ik}+{\dfrac{\theta}{{m}x_{ik}}}\leq 0,\nonumber&\\
		& \sum\limits_{d=1}^{D} v_{dk}-{\dfrac{1}{D}{{z_{dk}}}}\geq 0,\nonumber&\\
		& \sum\limits_{i=1}^{m} u_{ik}{x_{ij}}+\sum\limits_{d=1}^{d} v_{dk} {z_{dj}}+a \leq 0,\nonumber&\\
		& \theta+\sum\limits_{i=1}^{m} u_{ik}{x_{ik}}+\sum\limits_{d=1}^{d} v_{dk} {z_{dk}}+a \geq 1,\nonumber&\\
		&u_{ik}, \,\,\,\,v_{dk},\,\,\theta,\,\,a \mbox{ are unrestricted in sign}.&
		\end{flalign}
		\end{minipage}
		\hfill
		\begin{minipage}[b]{0.48\textwidth}
	\begin{flalign}\label{mod6}
		&\min = t \theta &\nonumber\\
		&\mbox{subject to}&\nonumber\\
		& \sum\limits_{i=1}^{m} u_{ik}+{\dfrac{\theta}{{m}x_{ik}}}\leq 0,\nonumber&\\
		& \sum\limits_{d=1}^{d} v_{dk}-{\dfrac{1}{D{z_{dk}}}}\geq 0,\nonumber&\\
		& \sum\limits_{i=1}^{m} u_{ik}{x_{ij}}+\sum\limits_{d=1}^{D} v_{dk} {z_{dj}}+a \leq 0,\nonumber&\\
		& \theta+\sum\limits_{i=1}^{m} u_{ik}{x_{ik}}+\sum\limits_{d=1}^{D} v_{dk} {z_{dk}}+a \geq 1,\nonumber&\\
		&\theta \geq 0,\,\,u_{ik}, \,\,\,\,v_{dk},\,\,a \mbox{ are unrestricted in sign}.&
		\end{flalign}
		\end{minipage}\\

		The optimal solution of the Models \ref{mod5} and \ref{mod6} is $\theta^*$. Since $\theta^* \leq 0$ is not possible in Model \ref{mod5} because $u_{ik}*=0$ in first and last constraint of Model \ref{mod5} because it violates the first constraint. Therefore, $\theta^* > 0$. Thus, dual Models \ref{mod5} and \ref{mod6} are equivalent, so the primal Models \ref{mod1} and \ref{mod3} are equivalent.

	\end{proof}

	\begin{definition}
		Let the positive parameter t be fixed at 1, the Model \ref{mod02} is the appropriate FSSBM model for the robust optimization:
		\begin{flalign}\label{mod02}
		&\max z_k^1= p_1+{\dfrac{1}{D}\left({\sum\limits_{d=1}^{d}{{S_{dk}}}/{z_{dk}}}\right)}\nonumber&\\
		&\mbox{subject to  }\Bigg\{\,\,\,p_1-\dfrac{1}{m}{\sum\limits_{i=1}^{m}{{S_{ik}^-}}/{x_{ik}}}\leq 1,\,\,\,p_1 x_{ik}-\sum\limits_{j=1}^{n}{x_{ij}\lambda_j'}-{{S_{ik}^-}}\geq 0,\,\,\forall i,\,\,\, p_1 z_{dk} - \sum\limits_{j=1}^{n}{z_{dj}\lambda_j'}+{{S_{dk}}}\leq 0,\,\,\,\,\forall d,\nonumber&\\
		& p_1-\sum\limits_{j=1}^{n}{\lambda_j'}=0,\,\,p_1 >0;\,\,\lambda_j'\geq 0, \,\forall j, \,\,\,{S_{ik}^-}\geq 0,\,\,\,{S_{dk}}\geq 0.\Bigg\}&
		\end{flalign}
		In Model \ref{mod02} all the input, intermediate and output data are uncertain.
	\end{definition}
	
	\begin{definition}
For $k=1,2,3,...,n$, the following Model (\ref{mod04}) is equal to Model (\ref{mod02}):
		\begin{flalign}\label{mod04}
		&\max w_k^1  \nonumber&\\
		&\mbox{subject to  }&\nonumber\\
		& \Bigg\{ w_k^1-p_1-{\dfrac{1}{D}\left({\sum\limits_{d=1}^{d}{{S_{dk}}}/{z_{dk}}}\right)} \leq 0,\,\,\,p_1-\dfrac{1}{m}{\sum\limits_{i=1}^{m}{{S_{ik}^-}}/{x_{ik}}}\leq 1,\,\,\,p_1 x_{ik}-\sum\limits_{j=1}^{n}{x_{ij}\lambda_j'}-{{S_{ik}^-}}\geq 0,\,\,\forall i,\,\,\, \nonumber&\\
		& p_1 z_{dk} - \sum\limits_{j=1}^{n}{z_{dj}\lambda_j'}+{{S_{dk}}}\leq 0,\,\,\,\,\forall d,\,\,\,p_1-\sum\limits_{j=1}^{n}{\lambda_j'}= 0,\,\,p_1 >0;\,\,\lambda_j'\geq 0, \,\forall j, \,\,\,{S_{ik}^-}\geq 0,\,\,\,{S_{dk}}\geq 0.\Bigg\}&
		\end{flalign}
	\end{definition}
	
	\begin{definition}
		The following Model (\ref{mod004}) is equivalent to Model (\ref{mod04}) for $k=1,2,3,...,n$:
		\begin{flalign}\label{mod004}
		&\max w_k^1  \nonumber&\\
		&\mbox{subject to  }\Bigg\{ w_k^1-p_1-{\dfrac{1}{D}\left({{\sum\limits_{d=1}^{D}{b_{dk}}}}\right)} \leq 0,\,\,\,p_1-\dfrac{1}{m}{\sum\limits_{i=1}^{m}{{a_{ik}}}}\leq 1,\,\,\,-(p_1 x_{ik}-\sum\limits_{j=1}^{n}{x_{ij}\lambda_j'}-{{S_{ik}^-}})\leq 0,\,\,\forall i,\nonumber&\\
		& p_1 z_{dk} - \sum\limits_{j=1}^{n}{z_{dj}\lambda_j'}+{{S_{dk}}}\leq 0,\,\,\,\,\forall d,\,\,\,{z_{dk}}{b_{dk}}-{S_{dk}} \leq 0,\,\,\,\,\forall d,\,\,\,{x_{ik}}{a_{ik}}-{S_{ik}^-} \leq 0,\,\,\,\,\forall i,\,\,\,p_1-\sum\limits_{j=1}^{n}{\lambda_j'}= 0,\,\,\,p_1 >0;
	\nonumber&\\
	&\lambda_j'\geq 0, \,\forall j, \,\,\,{S_{ik}^-}\geq 0,\,\,\,{S_{dk}}\geq 0.\Bigg\}&
		\end{flalign}
		Thus, Model (\ref{mod004}) is equivalent to Model (\ref{mod1}) for $DMU_k$. The Model (\ref{mod004}) is the appropriate first stage model to apply robust optimization in which input, and intermediate data are uncertain.
	\end{definition}
	
	\begin{definition} \label{def1}
		Let $\mbox{Eff}_k^1=1/w_k^1$ in Model \ref{mod004}. $\mbox{Eff}_k^1$ is also called the first stage SBM (FSSBM) efficiency (FSSBME) of $DMU_{k}$. Let $\mbox{Eff}_k^{1*}$ be the optimal value of $\mbox{Eff}_k^1$. Then $DMU_{k}$ is called FSSBM-efficient if $\mbox{Eff}_k^{1*}  = 1$; otherwise, $DMU_{k}$ is called FSSBM-inefficient.
	\end{definition}
	
	\begin{thm}\label{Thm.effi1}
	If $(p_1^{*}, \lambda_j^{*'}S_{ik}^{*-}S_{dk}^{*})$ is the optimal solution of the model \eqref{mod004}, then $w_{k}^{1} \ge 1$ and consequently $\mbox{Eff}_k^1 \in (0,~1]$.
	\end{thm}
\begin{proof}
	    For the fix $k$, assume that $p_{1}^{*}=1,~ a_{ik}^{*}=S_{ik}^{*-}=0,~\forall~ i,~ b_{dk}^{*}=S_{dk}^{*}=0,~\forall~ d,~ \lambda_{j}^{'*}=1 (j=k)$ otherwise $\lambda_{j}^{'*}=0$, then model \eqref{mod004} always gives a feasible solution.  Now, putting the value of $b_{dk}=0$ and multiplying $x_{ik}$ in the first constraint of the model \eqref{mod004} as:
	    \begin{eqnarray}\label{mod4.4}
	    p_{1} x_{ik} \ge w_{k}^{1}x_{ik}
	    \end{eqnarray}
	    where $p_{1}\ge 0$ is considered as the arbitrary value. From equation \eqref{mod4.4} and third constraint of the model \eqref{mod004}, we have
	    \begin{eqnarray}\label{mod4.5}
	    w_{k}^{1}x_{ik}-\sum_{j=1}^{n} x_{ij} \lambda_{j}^{'}-S_{ik}^{-} \ge 0.
	    \end{eqnarray}
	    Now, $\lambda_{j}^{'*}=1 (j=k)$ and $S_{ik}^{-}=0$  then equation \eqref{mod4.5} becomes: $(w_{k}^{1}-1) x_{ik}   \ge  0, \,\,   \implies~~ w_{k}^{1}   \ge   1.
	  $\\
	   Thus, by the definition \eqref{def1}, $\mbox{Eff}_k^1=\frac{1}{ w_{k}^{1}} \le 1.$
	
	  Again, if we consider $w_{k}^{1}=0,$ then  the  equation \eqref{mod4.5} becomes $-x_{ik} \ge 0$ which is not possible. Thus $\mbox{Eff}_k^1 \in (0,~1].$
	
	\end{proof}
Next, we provide the second stage SBM model as follows:
	\begin{flalign}\label{05}
	&\max z_{k}^2= p_2+\dfrac{1}{s_1+s_2}{\left({\sum\limits_{r_1=1}^{s_1}{{S_{r_1k}^+}}/{y_{r_1k}^D}}-{\sum\limits_{r_2=1}^{s_2}{{S_{r_2k}^+}}/{y_{r_2k}^U}}\right)}\nonumber&\\
	&\mbox{subject to   } \Bigg\{p_2-{\dfrac{1}{D}\left({\sum\limits_{d=1}^{D}{{S_{dk}}}/{z_{dk}}}\right)}=1,\,\,\,p_2 z_{dk} -\sum\limits_{j=1}^{n}{z_{dj}\lambda_j'}-{{S_{dk}}}=0,\,\,\,\,\forall d,\,\,\,p_2 y_{r_1k}^D -\sum\limits_{j=1}^{n}{y_{r_1j}^D\lambda_j'}+{{S_{r_1k}^+}}=0,\,\,\,\,\forall r_1,\nonumber&\\
	& -p_2 y_{r_2k}^U +\sum\limits_{j=1}^{n}{y_{r_2j}^U\lambda_j'}+{{S_{r_2k}^+}}=0,\,\,\,\,\forall r_1,\,\,\,p_2-\sum\limits_{j=1}^{n}{\lambda_j'}=0,\,\,p_2 >0;\,\,\lambda_j'\geq 0, \,\forall j, \,\,\,{S_{ik}^-}\geq 0,\,\,\,{S_{r_1k}^+}\geq 0,\,\,\,{S_{r_2k}^+}\geq 0.\Bigg\}&
	\end{flalign}
For $k=1,2,...,n,$ the following model \ref{07} is equivalent to model \ref{05}:
	
	\begin{flalign}\label{07}
	&\max w_k^2& \nonumber\\
	&\mbox{subject to  }\Bigg\{w_k^{2}-p_2-\dfrac{1}{s_1+s_2}{\left({\sum\limits_{r_1=1}^{s_1}{{S_{r_1k}^+}}/{y_{r_1k}^D}}-{\sum\limits_{r_2=1}^{s_2}{{S_{r_2k}^+}}/{y_{r_2k}^U}}\right)}\leq 0,\,\,\,p_2-{\dfrac{1}{D}\left({\sum\limits_{d=1}^{d}{{S_{dk}}}/{z_{dk}}}\right)}\leq 1,\nonumber&\\
	&p_2 z_{dk} -\sum\limits_{j=1}^{n}{z_{dj}\lambda_j'}-{{S_{dk}}}\geq 0,\,\,\,\,\forall d,\,\,\, p_2 y_{r_1k}^D -\sum\limits_{j=1}^{n}{y_{r_1j}^D\lambda_j'}+{{S_{r_1k}^+}}\leq 0,\,\,\,\,\forall r_1,\,\,\, -p_2 y_{r_2k}^U +\sum\limits_{j=1}^{n}{y_{r_2j}^U\lambda_j'}+{{S_{r_2k}^+}}\leq 0,\,\,\,\,\forall r_1,\,\,\,\nonumber&\\
	&p_2-\sum\limits_{j=1}^{n}{\lambda_j'}= 0,\,\,p_2 >0;\,\,\lambda_j'\geq 0, \,\forall j, \,\,\,{S_{r_1k}^+}\geq 0,\,\,\,{S_{r_2k}^+}\geq 0. \Bigg\}&
	\end{flalign}
The following model \ref{09} is equivalent to model \ref{07} for $k=1,2,...,n$:	
	\begin{flalign}\label{09}
	&\max w_k^2& \nonumber\\
	&\mbox{subject to  } \Bigg\{w_k^{2}-p_2-\dfrac{1}{(s_1+s_2)}{\left(\sum\limits_{r_1=1}^{s_1}{{c_{r_1k}}}-\sum\limits_{r_1=1}^{s_2}{{c_{r_2k}}}\right)}\leq 0,\,\,\,p_2-{\dfrac{1}{D}\left({{\sum\limits_{d=1}^{D}{{b_{dk}}}}}\right)}\leq 1,\nonumber&\\
	& p_2 z_{dk} -\sum\limits_{j=1}^{n}{z_{dj}\lambda_j'}-{{S_{dk}}} \geq 0,\,\,\,\,\forall d,\,\,\, p_2 y_{r_1k} -\sum\limits_{j=1}^{n}{y_{r_1j}\lambda_j'}+{{S_{r_1k}^+}}\leq 0,\,\,\,\,\forall r_1,\,\,\, -p_2 y_{r_2k} +\sum\limits_{j=1}^{n}{y_{r_2j}\lambda_j'}+{{S_{r_2k}^+}}\leq 0,\,\,\,\,\forall r_2,\nonumber&\\
	& p_2-\sum\limits_{j=1}^{n}{\lambda_j'}= 0,\,\, {y_{r_1k}}{c_{r_1k}}-{S_{r_1k}^+} \leq 0,\,\, {y_{r_2k}}{c_{r_2k}}-{S_{r_2k}^+} \leq 0,\nonumber&\\
	&{z_{dk}}{b_{dk}}-{S_{dk}} \leq 0,\,\,\,p_2 >0;\,\,\lambda_j'= 0, \,\forall j, \,\,\,{S_{dk}}\geq 0,\,\,\,{S_{r_1k}^+}\geq 0,\,\,\,{S_{r_2k}^+}\geq 0.\Bigg\}&
	\end{flalign}
Thus, Model \ref{09} is equivalent to model \ref{mod2} for $DMU_k$. The model \ref{09} is the appropriate second stage model to apply robust optimization in which intermediate and output data are uncertain.
\begin{definition}
		Let $\mbox{Eff}_k^2=1/w_k^2$ in Model \ref{09}. $\mbox{Eff}_k^2$ is also called the second stage SBM (SSSBM) efficiency (SSSBME) of $DMU_{k}$. Let $\mbox{Eff}_k^{2*}$ be the optimal value of $\mbox{Eff}_k^2$. Then $DMU_{k}$ is called SSSBM-efficient if $\mbox{Eff}_k^{2*}  = 1$; otherwise, $DMU_{k}$ is called SSSBM-inefficient.
	\end{definition}	
\begin{definition}
The overall efficiency, $\mbox{Eff}_k^{O}$, of $DMU_k$ is determined using first and second stage efficiencies $\mbox{Eff}_k^{O*}=\mbox{Eff}_k^{1*}\times \mbox{Eff}_k^{2*}$. $DMU_{k}$ is called efficient if $\mbox{Eff}_k^{O*}  = 1$, i.e. iff $\mbox{Eff}_k^{2*}=1$ and $\mbox{Eff}_k^{2*}=1$; otherwise, $DMU_{k}$ is called inefficient.
	\end{definition}	
	
	\begin{thm}\label{Thm.effi2}
	If $(p_2^{*}, \lambda_j^{*'},~ S_{dk}^{*},~S_{r_{1}k}^{*+},~S_{r_{2}k}^{*+})$ is the optimal solution of the model \eqref{09}, then $w_{k}^{2} \ge 1$ and consequently $\mbox{Eff}_k^2 \in (0,~1]$.
	\end{thm}
	\begin{proof}
	  The proof of this theorem is similar to Theorem-\ref{Thm.effi1}.
	\end{proof}
	
	\subsection{Robust model of two-stage SBM}
	Robust optimization is often used when the data information is given in the form of an uncertainty set. In the SBM model \eqref{mod004}, we consider the affine function of the input parameters $x_{ik}, x_{ij}, z_{dk}$ and $z_{dj}$ are $x_{ik} = x_{ik}^{0}+\sum_{\ell=1}^{L}\zeta_{\ell}x_{ik}^{\ell}$, $x_{ij} = x_{ij}^{0}+\sum_{\ell=1}^{L}\zeta_{\ell}x_{ij}^{\ell}$, $z_{dk} = z_{dk}^{0}+\sum_{\ell=1}^{L}\zeta_{\ell}z_{dk}^{\ell}$ and $z_{dj} = z_{dj}^{0}+\sum_{\ell=1}^{L}\zeta_{\ell}z_{dj}^{\ell}$ where $x_{ik}^{0}, x_{ij}^{0}, z_{dk}^{0}, z_{dj}^{0}$ are the nominal data and $x_{ik}^{\ell}, x_{ij}^{\ell}, z_{dk}^{\ell}, z_{dj}^{\ell}$ are possible deviations, and $\zeta \in \mathcal{R}^{L}$ is the perturbation vector. Thus, the uncertainty sets for the fix $k$ are defined as:
	\begin{eqnarray}
	\mathcal{U}_{x} =\left\{ (x_{ik}, x_{ij}) ~|~ x_{ik}^{0}+\sum_{\ell=1}^{L}\zeta_{\ell}x_{ik}^{\ell}, ~ x_{ij}^{0}+\sum_{\ell=1}^{L}\zeta_{\ell}x_{ij}^{\ell}, ~~\zeta \in \mathcal{Z}_{i} ~\forall~i\right\}
	\end{eqnarray}
	\begin{eqnarray}
	\mathcal{U}_{d} =\left\{ (z_{dk}, z_{dj}) ~|~ z_{dk}^{0}+\sum_{\ell=1}^{L}\zeta_{\ell}z_{dk}^{\ell}, ~ z_{dj}^{0}+\sum_{\ell=1}^{L}\zeta_{\ell}z_{dj}^{\ell}, ~~\zeta \in \mathcal{Z}_{d}, ~\forall~d\right\}
	\end{eqnarray}
	where $\mathcal{Z}_{i},~\mathcal{Z}_{d}\subset \mathcal{R}^{L}$ are the user-defined uncertainty sets. Accordingly, the robust formulation of the first stage SBM model \eqref{mod004} is given by:
	\begin{flalign}
	& \max ~~~w_k^1  &\nonumber \\
	& {\hbox{subject to }}\Bigg\{w_k^1-p_1-{\dfrac{1}{D}\left({{\sum\limits_{d=1}^{D}{b_{dk}}}}\right)}  \leq  0, \,\,\,\nonumber p_1-\dfrac{1}{m}{\sum\limits_{i=1}^{m}{{a_{ik}}}} \leq  1, \,\,\, -(p_1 x_{ik}-\sum\limits_{j=1}^{n}{x_{ij}\lambda_j'}-{{S_{ik}^-}})  \leq  0,~~\forall~ i,~~& \\
	& \nonumber\left(x_{ik},~x_{ij}\right)\in \mathcal{U}_{x},
	\,\,\ p_1 z_{dk} - \sum\limits_{j=1}^{n}{z_{dj}\lambda_j'}+{{S_{dk}}} \leq  0,~~\forall~ d,~~\left(z_{dk},~z_{dj}\right)\in \mathcal{U}_{z},\,\,\,z_{dk}b_{dk}-S_{dk}  \leq 0,~~~\forall~ d,~~z_{dk}\in \mathcal{U}_{z},& \\
	&  x_{ik}a_{ik}-S_{ik}^{-}  \leq  0,~~~\forall ~i,~~x_{ik}\in \mathcal{U}_{x},\,\,p_1-\sum\limits_{j=1}^{n}{\lambda_j'}  =  0,
p_1 >0;\,\,\lambda_j'\geq 0, \,\forall j, \,\,\,{S_{ik}^-}\geq 0,\,\,\,{S_{rk}^+}\geq 0,\,\,\,{S_{dk}}  \geq  0.\Bigg\}\label{mod1.13}&
	\end{flalign}
	\subsection{\bf Ellipsoidal uncertainty set}
	Suppose that the support of uncertain data belongs to the ellipsoidal uncertainty set. Then, to obtain robust equivalence, we present the following theorem:
	\begin{thm}\label{Theorem1DEA}
	Tractable formulation of first stage robust SBM model \eqref{mod1.13} under the ellipsoidal uncertainty sets $\mathcal{Z}_{i}=\left\{\zeta~|~\|\zeta\|_{2}\le \Omega_{i},~\forall ~i\right\}$ and $\mathcal{Z}_{d}=\left\{\zeta~|~\|\zeta\|_{2}\le \Omega_{d},~\forall ~d\right\}$  is equivalent to the polynomial size second order conic optimization problem:
	
		\begin{flalign}
		& \max ~~~w_k^1 & \nonumber\\
		& \hbox{subject to } \Bigg\{w_k^1-p_1-{\dfrac{1}{D}\left({{\sum\limits_{d=1}^{D}{b_{dk}}}}\right)} \leq  0, \,\,\,p_1-\dfrac{1}{m}{\sum\limits_{i=1}^{m}{{a_{ik}}}}  \leq  1,&\nonumber\\
		& -p_1 x_{ik}^{0}+\sum\limits_{j=1}^{n}x_{ij}^{0}\lambda_j'+S_{ik}^{-}+\Omega_{i}\sqrt{\sum_{\ell=1}^{L}\left(\sum\limits_{j=1}^{n}
			\left(x_{ij}^{\ell}\lambda_j'\right)-p_1 x_{ik}^{\ell}\right)^{2} }  \leq  0,~~\forall~ i, & \nonumber\\
		& p_1 z_{dk}^{0} - \sum_{j=1}^{n}{z_{dj}^{0}\lambda_j'}+{{S_{dk}}}+ \Omega_{d}\sqrt{\sum_{\ell=1}^{L}\left(p_{1}z_{dk}^{\ell}-\sum\limits_{j=1}^{n}
			z_{dj}^{\ell}\lambda_j'\right)^{2} } \leq  0,~~\forall~ d,& \nonumber\\
		& z_{dk}^{0}b_{dk}+\Omega_{d}\sqrt{\sum_{\ell=1}^{L}\left(z_{dk}^{\ell}b_{dk}\right)^{2}}-S_{dk}  \leq  0,~~~\forall~ d,\,\, x_{ik}^{0}a_{ik}+\Omega_{i}\sqrt{\sum_{\ell=1}^{L}\left(x_{ik}^{\ell}a_{ik}\right)^{2}}-S_{ik}^{-} \leq  0,~~~\forall ~i,& \nonumber\\
		& p_1-\sum\limits_{j=1}^{n}{\lambda_j'}  =  0,
		p_1 >0;\,\,\lambda_j'\geq 0, \,\forall j, \,\,\,{S_{ik}^-}\geq 0,\,\,\,{S_{dk}} \geq  0.\Bigg\} \label{mod1.14}&
		\end{flalign}
	\end{thm}
	\begin{proof}
		Proof of this theorem is divided into three steps which are depicted as follows:\\
		
		\textbf{Step 1} (\textit{Worst case scenario}) The third constraint of the first stage robust SBM model \eqref{mod1.13}  can be articulated by applying the max-min regret approach to define the following worst-case formulation for the fixed $k$ is:
		\begin{eqnarray}
		\nonumber && -p_1 x_{ik}^{0}+\sum\limits_{j=1}^{n}x_{ij}^{0}\lambda_j'+S_{ik}^{-}+ \max_{\zeta \in \mathcal{Z}_{i}}\left\{\sum_{\ell=1}^{L}\zeta_{\ell}\left(\sum_{j=1}^{n} \left(x_{ij}^{\ell}\lambda_j'\right)-p_1 x_{ik}^{\ell}\right) \right\} \le 0, ~~~~\forall~i,\\ \label{proof1}
		&&  -p_1 x_{ik}^{0}+\sum\limits_{j=1}^{n}x_{ij}^{0}\lambda_j'+S_{ik}^{-}+ \gamma_{i}(\lambda_j', p_{1})\le 0 ~~~~\forall~i,
		\end{eqnarray}
		where $\gamma_{i}(\lambda_j', p_{1})=  \max\limits_{\zeta \in \mathcal{Z}_{i}}\left\{\sum_{\ell=1}^{L}\zeta_{\ell}\left(\sum_{j=1}^{n} \left(x_{ij}^{\ell}\lambda_j'\right)-p_1 x_{ik}^{\ell}\right) \right\}~~\forall~i,$ is a protection function against all the uncertain input and output data.\\
		
		\textbf{Step 2} (\textit{Dual approach}) The protection function, defined in Step-1, is clearly a linear objective function with a bounded quadratic constraint, implying that it will provide a feasible solution. Therefore, according to the strong duality approach, Its dual also provide a feasible solution. As a result, the Lagrangian function of the inner problem of the protection function is provided by:
		\begin{eqnarray*}
			L(\zeta, \lambda_{i})=\sum_{\ell=1}^{L}\zeta_{\ell}\left(\sum_{j=1}^{n} \left(x_{ij}^{\ell}\lambda_j'\right)-p_1 x_{ik}^{\ell}\right)-\lambda_{i}\left(\left(\sum_{\ell=1}^{L}\zeta_{\ell}^2\right)^{\frac{1}{2}}- \Omega_{i}\right)
		\end{eqnarray*}
		where $ \lambda_{i}$ denotes the dual variables associated with the inner problem, thereby, the optimal value of the Lagrange dual function is defined as:
$
			\inf_{\zeta} ~~L(\zeta, \lambda_{i})
			 =  \Omega_{i}\sqrt{\sum_{\ell=1}^{L}\left(\sum\limits_{j=1}^{n}
				\left(x_{ij}^{\ell}\lambda_j'\right)-p_1 x_{ik}^{\ell}\right)^{2} }=\gamma_{i}\big(\lambda_j', p_{1} \big)
$
		
		\textbf{Step 3} (\textit{Robust counterpart}) By substituting the value of $\gamma_{i}\big(\lambda_j', p_{1} \big)$ in the \eqref{proof1}, obtain the following equation:
		\begin{eqnarray}
		-p_1 x_{ik}^{0}+\sum\limits_{j=1}^{n}x_{ij}^{0}\lambda_j'+S_{ik}^{-}+\Omega_{i}\sqrt{\sum_{\ell=1}^{L}\left(\sum\limits_{j=1}^{n}
			\left(x_{ij}^{\ell}\lambda_j'\right)-p_1 x_{ik}^{\ell}\right)^{2} } & \leq & 0,~~\forall~ i,
		\end{eqnarray}
		which is the desired result of the third constraint of the first stage robust SBM model. Similarly, we can obtain the tractable formulation of the other uncertain constraints of the first stage robust SBM model \eqref{mod1.13}.
	\end{proof}
	\begin{definition}
		Let ${E^e}_k^{1}=1/w_k^1$ be the first stage efficiency in ellipsoidal perturbation in Model (\ref{mod1.14}). Let ${E^e}_k^{1*}$ be the optimal value of ${E^e}_k^{1}$. Then $DMU_{k}$ is called first-stage efficient in ellipsoidal perturbation if ${E^e}_k^{1*}  = 1$; otherwise, $DMU_{k}$ is called first-stage inefficient in ellipsoidal perturbation.
	\end{definition}	

	\begin{thm}\label{Thm.effi3}
	If $(p_1^{*},~ \lambda_j^{*'},~S_{ik}^{*-},~S_{dk}^{*},~a_{ik}^{*},~b_{dk}^{*})$ is the optimal solution of the model \eqref{mod1.14}, then $w_{k}^{1} \ge 1$ and consequently ${E^e}_k^{1} \in (0,~1]$.
	\end{thm}
	\begin{proof}
	    This theorem's proof is similar to Theorem-\ref{Thm.effi1}.
	\end{proof}
	
	\begin{thm}\label{thm3}
		Computationally tractable formulation of the second stage robust SBM model \eqref{09} with the ellipsoidal uncertainty sets $\mathcal{Z}_{r_1}=\left\{\zeta~|~\|\zeta\|_{2}\le \Omega_{r_1},~\forall ~r_1\right\}$, $\mathcal{Z}_{r_2}=\left\{\zeta~|~\|\zeta\|_{2}\le \Omega_{r_2},~\forall ~r_2\right\}$ and $\mathcal{Z}_{d}=\left\{\zeta~|~\|\zeta\|_{2}\le \Omega_{d},~\forall ~d\right\}$  is equivalent to the polynomial size second order conic optimization problem:
		\begin{flalign}
		&\max w_k^2& \nonumber\\
		&\mbox{subject to  }\Bigg\{w_k^{2}-p_2-\dfrac{1}{(s_1+s_2)}{\left(\sum\limits_{r_1=1}^{s_1}{{c_{r_1k}}}-\sum\limits_{r_1=1}^{s_2}{{c_{r_2k}}}\right)}\leq 0,\,\,\,p_2-{\dfrac{1}{D}\left({{\sum\limits_{d=1}^{D}{{b_{dk}}}}}\right)}\leq 1,\nonumber&\\
		& -p_2 z_{dk}^{0}+\sum\limits_{j=1}^{n}z_{dj}^{0}\lambda_j'+S_{dk}+\Omega_{d}\sqrt{\sum_{\ell=1}^{L}\left(\sum\limits_{j=1}^{n}
			\left(z_{dj}^{\ell}\lambda_j'\right)-p_2 z_{dk}^{\ell}\right)^{2} }  \leq  0,~~\forall~ d, & \nonumber
				\end{flalign}
		\begin{flalign}
		& p_2 y_{r_1k}^{0} -\sum\limits_{j=1}^{n}y_{r_1j}^{0}\lambda_j'+{{S_{r_1k}^+}}+\Omega_{r_1} \sqrt{\sum_{\ell=1}^{L}\left(p_2 y_{r_1k}^{\ell}-\sum\limits_{j=1}^{n}y_{r_1j}^{\ell}\lambda_j'\right)^{2}}\leq 0,\,\,\,\,\forall r_1,\nonumber&\\
		& -p_2 y_{r_2k}^{0} +\sum\limits_{j=1}^{n}y_{r_2j}^{0}\lambda_j'+{{S_{r_2k}^+}}+\Omega_{r_2}\sqrt{\sum_{\ell=1}^{L}\left(\sum\limits_{j=1}^{n}
			\left(y_{r_2j}^{\ell}\lambda_j'\right)-p_2 y_{r_2k}^{\ell}\right)^{2} }\leq 0,\,\,\,\,\forall r_2,\nonumber&\\
		& p_2-\sum\limits_{j=1}^{n}{\lambda_j'}= 0,\,\,\, y_{r_1k}^{0}{c_{r_1k}}-{S_{r_1k}^+}+ \Omega_{r_1} \sqrt{\sum_{\ell=1}^{L}\left( y_{r_1k}^{\ell}c_{r_1k}\right)^{2}} \leq 0,\,\,\,y_{r_2k}^{0}{c_{r_2k}}-{S_{r_2k}^+}+ \Omega_{r_2} \sqrt{\sum_{\ell=1}^{L}\left( y_{r_2k}^{\ell}c_{r_2k}\right)^{2}} \leq 0,\nonumber&\\
		&{z_{dk}^{0}}{b_{dk}}-{S_{dk}}+ \Omega_{d} \sqrt{\sum_{\ell=1}^{L}\left( z_{dk}^{\ell}b_{dk}\right)^{2}} \leq 0,\,\,\,p_2 >0;\,\,\lambda_j'= 0, \,\forall j, \,\,\,{S_{dk}}\geq 0,\,\,\,{S_{r_1k}^+}\geq 0,\,\,\,{S_{r_2k}^+}\geq 0. \Bigg\}\label{Secondstageellip}&
		\end{flalign}
	\end{thm}
\begin{proof}
The proof follows the same pattern as Theorem \ref{Theorem1DEA}.
\end{proof}	
	\begin{definition}
		Let ${E^e}_k^{2}=1/w_k^2$ be the second-stage efficiency in ellipsoidal perturbation in Model (\ref{Secondstageellip}). Let ${E^e}_k^{2*}$ be the optimal value of ${E^e}_k^{2}$. Then $DMU_{k}$ is called second-stage efficient in ellipsoidal perturbation if ${E^e}_k^{2*}  = 1$; otherwise, $DMU_{k}$ is called second-stage inefficient in ellipsoidal perturbation.
	\end{definition}
\begin{definition}
The overall efficiency, ${E^e}_k^{O}$, of $DMU_k$ is determined using first and second stage efficiencies ${E^e}_k^{O*}={E^e}_k^{1*}\times {E^e}_k^{2*}$. $DMU_{k}$ is called efficient if ${E^e}_k^{O*}  = 1$, i.e. iff ${E^e}_k^{2*}=1$ and ${E^e}_k^{2*}=1$; otherwise, $DMU_{k}$ is called inefficient in ellipsoidal perturbation.
	\end{definition}
	
	\begin{thm}\label{Thm.effi4}
	If $(p_2^{*},~ \lambda_j^{*'},~S_{r_{1}k}^{*+},~~S_{r_{2}k}^{*+},~S_{dk}^{*},~c_{r_1k}^{*},~c_{r_2k}^{*},~b_{dk}^{*})$ is the optimal solution of the model \eqref{Secondstageellip}, then $w_{k}^{2} \ge 1$ and consequently ${E^e}_k^{2} \in (0,~1]$.
	\end{thm}
	\begin{proof}
	    The proof of this theorem is similar to Theorem-\ref{Thm.effi1}.
	\end{proof}
	
	\subsection{\bf Polyhedral uncertainty set}
	Here, we consider that the support of uncertain data is given in a polyhedral uncertainty set. Thus, to obtain robust equivalence, we present the following theorem:
	
	\begin{thm}\label{theoremR1}
		Computationally tractable formulation of the  first stage robust SBM model \eqref{mod1.13} under the polyhedral uncertainty sets $\mathcal{Z}_{i}=\left\{\zeta~|~H_{i}\zeta +q_{i} \ge 0,~\forall ~i\right\}$ and $\mathcal{Z}_{d}=\left\{\zeta~|~H_{d}\zeta +q_{d}\ge 0,~\forall ~d\right\}$ where $H_{i}, H_{d} \in \mathcal{R}^{K \times L}, \zeta \in \mathcal{R}^{L} $ and $q_{i}, q_{d} \in \mathcal{R}^{K}$ is equivalent to the following linear optimization problem:\\
		\begin{flalign}
		& \max w_k^1 & \nonumber\\
	&	\hbox{subject to } \Bigg\{w_k^1-p_1-{\dfrac{1}{D}\left({{\sum\limits_{d=1}^{D}{b_{dk}}}}\right)}  \leq  0, \,\,\,p_1-\dfrac{1}{m}{\sum\limits_{i=1}^{m}{{a_{ik}}}}  \leq  1,\,\,\,-p_1 x_{ik}^{0}+\sum\limits_{j=1}^{n}x_{ij}^{0}\lambda_j'+S_{ik}^{-}+\sum_{k=1}^{K}q_{i}^{k}\nu_{k}  \leq  0~~\forall~ i,& \nonumber\\
		& \sum_{k=1}^{K}h_{\ell k}^{i}\nu_{k}+\sum\limits_{j=1}^{n}\left(x_{ij}^{\ell}\lambda_j'\right)-p_1 x_{ik}^{\ell} = 0,~~\forall~ i, \ell,\,\,\,p_1 z_{dk}^{0} - \sum_{j=1}^{n}{z_{dj}^{0}\lambda_j'}+{{S_{dk}}}+ \sum_{k=1}^{K} q_{d}^{k}\nu_{k}^1   \leq  0,~~\forall~ d,& \nonumber\\
		& \sum_{k=1}^{K}h_{\ell k}^{d}\nu_{k}^1+p_{1}z_{dk}^{\ell}-\sum\limits_{j=1}^{n}
	 z_{dj}^{\ell}\lambda_j'  =  0,~~\forall~ i, \ell,\,\,\,z_{dk}^{0}b_{dk}-S_{dk}+\sum_{k=1}^{K}q_{d}^{k}\nu_{k}^{2}  \leq  0,~~~\forall~ d,\,\,\,\sum_{k=1}^{K}h_{\ell k}^{d}\nu_{k}^2+z_{dk}^{\ell}b_{dk} =  0,~~~\forall~ d, \ell,& \nonumber\\
		& x_{ik}^{0}a_{ik}-S_{ik}^{-}+\sum_{k=1}^{K}q_{i}^{k}\nu_{k}^{3}   \leq  0,~~~\forall ~i,\,\,\,\sum_{k=1}^{K}h_{\ell k}^{i}\nu_{k}^3+x_{ik}^{\ell}a_{ik} =  0,~~~\forall~ i, \ell,\,\,\, p_1-\sum\limits_{j=1}^{n}{\lambda_j'} =  0,\,\,\, p_1 >0;& \nonumber\\&\lambda_j'\geq 0 \,\forall j, \,\,\,{S_{dk}} \geq  0 \,\forall d, \Bigg\}\label{mod1.15}&
		\end{flalign}
		where $h_{\ell k}^{i} \in H_{i}, h_{\ell k}^{d} \in H_{d}$ and $\nu_{k}, \nu_{k}^1, \nu_{k}^2, \nu_{k}^3 $ are the auxiliary dual variables.
	\end{thm}
	
	\begin{proof}
		Again, proof of this theorem is divided into three steps which are depicted as follows:\\
		
		\textbf{Step 1} (\textit{Worst-case scenario}) The third constraint of the  first stage robust SBM model \eqref{mod1.13} can be formulated by applying  the max-min regret approach to define the following worst-case formulation, for the fixed $k$ as:
		\begin{eqnarray}
		\nonumber && -p_1 x_{ik}^{0}+\sum\limits_{j=1}^{n}x_{ij}^{0}\lambda_j'+S_{ik}^{-}+ \max_{\zeta | H_{i}\zeta +q_{i} \ge 0}\left\{\sum_{\ell=1}^{L}\zeta_{\ell}\left(\sum_{j=1}^{n} \left(x_{ij}^{\ell}\lambda_j'\right)-p_1 x_{ik}^{\ell}\right) \right\} \le 0,~~~~\forall~i,\\ \label{proof2}
		&&  -p_1 x_{ik}^{0}+\sum\limits_{j=1}^{n}x_{ij}^{0}\lambda_j'+S_{ik}^{-}+ \gamma_{i}(\lambda_j', p_{i})\le 0, ~~~\forall~i,
		\end{eqnarray}
		where $\gamma_{i}(\lambda_j', p_{i})=  \sup \limits_{\zeta | H_{i}\zeta +q_{i} \ge 0}\left\{\sum_{\ell=1}^{L}\zeta_{\ell}\left(\sum_{j=1}^{n} \left(x_{ij}^{\ell}\lambda_j'\right)-p_1 x_{ik}^{\ell}\right) \right\}$ is a protection function against all the uncertain parameters.\\
		
		\textbf{Step 2} (\textit{Dual approach}) The protection function is clearly a linear objective function with a bounded linear constraint that yields a feasible solution. Thus, according to the strong duality approach, its dual formulation also provides a feasible solution; as a result, the dual formulation of the underlying inner problem of the protection function is given as:
		\begin{eqnarray*}
			\inf_{\nu} \left\{ \sum_{k=1}^{K}q_{i}^{k}\nu_{k}~\Big|~ \sum_{k=1}^{K}h_{\ell k}^{i}\nu_{k}+\sum\limits_{j=1}^{n}\left(x_{ij}^{\ell}\lambda_j'\right)-p_1 x_{ik}^{\ell} =  0,~~\forall~ i, \ell,\right \}
		\end{eqnarray*}
		where $ \nu \in \mathcal{R}^{K}$ are the dual variables associated with the inner problem of the protection function. Thus, equation \eqref{proof2} becomes:
		\begin{eqnarray}\label{proof3}
		-p_1 x_{ik}^{0}+\sum\limits_{j=1}^{n}x_{ij}^{0}\lambda_j'+S_{ik}^{-}+\inf_{\nu} \left\{ \sum_{k=1}^{K}q_{i}^{k}\nu_{k}~\Big|~ \sum_{k=1}^{K}h_{\ell k}^{i}\nu_{k}+\sum\limits_{j=1}^{n}\left(x_{ij}^{\ell}\lambda_j'\right)-p_1 x_{ik}^{\ell} =  0,~~\forall~ i, \ell,\right \} \le 0.
		\end{eqnarray}
		
		\textbf{Step 3} (\textit{Robust counterpart}) Henceforth, at least one value of $\nu$ is sufficient to hold equation \eqref{proof3}. Thus, we can omit the infimum term from the equation. Now, the final robust counterpart is given by:
		\begin{eqnarray}
		\nonumber -p_1 x_{ik}^{0}+\sum\limits_{j=1}^{n}x_{ij}^{0}\lambda_j'+S_{ik}^{-}+\sum_{k=1}^{K}q_{i}^{k}\nu_{k} & \leq & 0~~\forall~ i,\\
		\nonumber \sum_{k=1}^{K}h_{\ell k}^{i}\nu_{k}+\sum\limits_{j=1}^{n}\left(x_{ij}^{\ell}\lambda_j'\right)-p_1 x_{ik}^{\ell} & = & 0,~~\forall~ i, \ell,
		\end{eqnarray}
		Similarly, we can obtain the tractable formulation of the other uncertain constraints of the first stage robust SBM model \eqref{mod1.13}.
	\end{proof}
		\begin{definition}
		Let ${E^p}_k^{1}=1/w_k^1$ be the first stage efficiency in polyhedral perturbation in Model \eqref{mod1.15}. Let ${E^p}_k^{1*}$ be the optimal value of ${E^p}_k^{1}$. Then $DMU_{k}$ is called first-stage efficient in polyhedral perturbation if ${E^p}_k^{1*}  = 1$; otherwise, $DMU_{k}$ is called first-stage inefficient in polyhedral perturbation.
	\end{definition}
	\begin{thm}\label{Thm.effi5}
	If $(p_1^{*},~ \lambda_j^{*'},~S_{ik}^{*-},~S_{dk}^{*},v_{k}^{*},~v_{k}^{1*},~v_{k}^{2*}, ~v_{k}^{3*})$ is the optimal solution of the model \eqref{mod1.15}, then $w_{k}^{1} \ge 1$ and consequently ${E^p}_k^{1} \in (0,~1]$.
	\end{thm}
	\begin{proof}
	    The proof of this theorem is similar to Theorem-\ref{Thm.effi1}.
	\end{proof}

	Next, we provide a theorem to compute the tractable formulation of the second stage robust SBM model when input and output data are given in the form of a polyhedral uncertainty set.
	
	\begin{thm}
		
		Computationally tractable formulation of the second stage robust SBM model \eqref{09} with the Polyhedral uncertainty sets $\mathcal{Z}_{r_1}=\left\{\zeta~|~H_{r_{1}}\zeta + q_{r_{1}}\ge 0,~\forall ~r_1\right\}$, $\mathcal{Z}_{r_2}=\left\{\zeta~|H_{r_{2}}\zeta + q_{r_{2}}\ge 0,~\forall ~r_2\right\}$ and $\mathcal{Z}_{d}=\left\{\zeta~|~H_{d}\zeta + q_{d}\ge 0,~\forall ~d\right\}$ is equivalent to the polynomial-size second-order conic optimization problem:
		\begin{flalign}\label{012}
			&\max w_k^2& \nonumber\\
			&\mbox{subject to  } \Bigg\{ w_k^{2}-p_2-\dfrac{1}{(s_1+s_2)}{\left(\sum\limits_{r_1=1}^{s_1}{{c_{r_1k}}}-\sum\limits_{r_2=1}^{s_2}{{c_{r_2k}}}\right)}\leq 0,\,\,p_2-{\dfrac{1}{D}\left({{\sum\limits_{d=1}^{D}{{b_{dk}}}}}\right)}\leq 1,& \nonumber\\
			&-p_2 z_{dk}^{0}+\sum\limits_{j=1}^{n}z_{dj}^{0}\lambda_j'+S_{dk}+ \sum\limits_{k_{1}=1}^{K} q_{d}^{k_{1}} v_{k_{1}} \leq  0,~~\forall~ d,  \,\,\sum\limits_{k_{1}=1}^{K} h_{\ell k_{1}}^{d} v_{k_{1}}+  \sum\limits_{j=1}^{n} z_{dj}^{\ell}\lambda_j'-p_2 z_{dk}^{\ell} = 0,~~\forall~ d, \ell,  \nonumber &\end{flalign}
		\begin{flalign}
			& p_2 y_{r_1k}^{0} -\sum\limits_{j=1}^{n}y_{r_1j}^{0}\lambda_j'+{{S_{r_1k}^+}}+ \sum\limits_{k_{1}=1}^{K} q_{r_{1}}^{k_{1}} v_{k_{1}}^{1} \leq 0,\,\,\,\,\forall r_1,\,\,\,\sum\limits_{k_{1}=1}^{K} h_{\ell k_{1}}^{r_{1}} v_{k_{1}}^{1}+ p_2 y_{r_1k}^{\ell}-\sum\limits_{j=1}^{n}y_{r_1j}^{\ell}\lambda_j' = 0,~~\forall~ r_{1}, \ell, &\nonumber	\\
			& -p_2 y_{r_2k}^{0} +\sum\limits_{j=1}^{n}y_{r_2j}^{0}\lambda_j'+{{S_{r_2k}^+}}+\sum\limits_{k_{1}=1}^{K} q_{r_{2}}^{k_{1}} v_{k_{1}}^{2} \leq 0,\,\,\,\,\forall r_2,\,\,\, \sum\limits_{k_{1}=1}^{K} h_{\ell k_{1}}^{r_{2}} v_{k_{1}}^{2}+ \sum\limits_{j=1}^{n}y_{r_2j}^{\ell}\lambda_j'-p_2 y_{r_2k}^{\ell}= 0,~~\forall~ r_{2}, \ell, \nonumber & \\
			& p_2-\sum\limits_{j=1}^{n}{\lambda_j'}= 0,\,\,\,  y_{r_1k}^{0}{c_{r_1k}}-{S_{r_1k}^+}+ \sum\limits_{k_{1}=1}^{K} q_{r_{1}}^{k_{1}} v_{k_{1}}^{3}  \leq 0,\,\,\,\,\forall r_1,\,\,\,  \sum\limits_{k_{1}=1}^{K} h_{\ell k_{1}}^{r_{1}} v_{k_{1}}^{3}+ y_{r_1k}^{\ell}c_{r_1k} = 0, \,\,\,\,\forall r_1, \ell,\nonumber &\\
			&   y_{r_2k}^{0}{c_{r_2k}}-{S_{r_2k}^+}+  \sum\limits_{k_{1}=1}^{K} q_{r_{2}}^{k_{1}} v_{k_{1}}^{4} \leq 0,\,\,\,\,\forall r_2,\,\,\, \sum\limits_{k_{1}=1}^{K} h_{\ell k_{1}}^{r_{2}} v_{k_{1}}^{4}+ y_{r_2k}^{\ell}c_{r_2k}= 0, \,\,\,\,\forall r_2, \ell, \nonumber& \\
			&z_{dk}^{0}{b_{dk}}-{S_{dk}}+ \sum\limits_{k_{1}=1}^{K} q_{d}^{k_{1}} v_{k_{1}}^{5} \Omega_{d} \leq 0, \,\,\,\,\forall d,\,\,\,\sum\limits_{k_{1}=1}^{K} h_{\ell k_{1}}^{d} v_{k_{1}}^{5}+z_{dk}^{\ell}b_{dk} =0 \,\,\,\,\forall d, \ell, \,\,\,p_2 >0;\,\,\lambda_j'= 0, \,\forall j, \nonumber& \\&{S_{dk}}\geq 0,\,\,\,{S_{r_1k}^+}\geq 0,\,\,\,{S_{r_2k}^+}, v_{k_{1}}, v_{k_{1}}^{1}, v_{k_{1}}^{2},v_{k_{1}}^{3},v_{k_{1}}^{4},v_{k_{1}}^{5}\geq 0. \Bigg\}&
		\end{flalign}
	\end{thm}
	\begin{proof}
The proof follows the same pattern as Theorem \ref{theoremR1}.
	\end{proof}
	\begin{definition}
		Let ${E^p}_k^{2}=1/w_k^2$ be the second-stage efficiency in polyhedral perturbation in Model (\ref{012}). Let ${E^p}_k^{2*}$ be the optimal value of ${E^p}_k^{2}$. Then $DMU_{k}$ is called second-stage efficient in polyhedral perturbation if ${E^p}_k^{2*}  = 1$; otherwise, $DMU_{k}$ is called second-stage inefficient in polyhedral perturbation.
	\end{definition}
\begin{definition}
The overall efficiency, ${E^p}_k^{O}$, of $DMU_k$ is determined using first and second stage efficiencies ${E^p}_k^{O*}={E^p}_k^{1*}\times {E^p}_k^{2*}$. $DMU_{k}$ is called efficient if ${E^p}_k^{O*}  = 1$, i.e. iff ${E^p}_k^{2*}=1$ and ${E^p}_k^{2*}=1$; otherwise, $DMU_{k}$ is called inefficient in ellipsoidal perturbation.

	\end{definition}
	
	\begin{thm}\label{Thm.effi6}
	If $(p_2^{*},~ \lambda_j^{*'},{S_{r_1k}^{*+}},~S_{r_2k}^{*+},~S_{dk}^{*},v_{k_{1}}^{*}, v_{k_{1}}^{*1}, v_{k_{1}}^{*2},v_{k_{1}}^{*3},v_{k_{1}}^{*4},v_{k_{1}}^{*5})$ is the optimal solution of the model (\ref{012}), then $w_{k}^{2} \ge 1$ and consequently ${E^p}_k^{2} \in (0,~1]$.
	\end{thm}
	\begin{proof}
	    The proof of this theorem is similar to Theorem-\ref{Thm.effi1}.
	\end{proof}

	\subsection{\bf Budget uncertainty set}
	In this case, we assume that all the input and output $x_{ij}(\zeta),~ x_{ij}(\zeta),~  z_{dj}(\zeta), \hbox{ and } z_{dj}(\zeta)$ are affinely depend on some exogenous uncertainty $\zeta \in \mathcal{Z}=[-1, 1]^{L}$. If $\zeta=0$, then SBM model becomes crisp/nominal. Now, we rewrite the uncertain constraint of the SBM model \eqref{mod1.13} as:
	\begin{eqnarray*}
	&\nonumber -p_1 x_{ik}(\zeta)+\sum\limits_{j=1}^{n}{x_{ij}(\zeta)\lambda_j'}+{{S_{ik}^-}}  \leq  0,~~\zeta \in \mathcal{U}_{\Gamma_{i}}~~\forall~ i,
	&\\
	& \nonumber p_1 z_{dk}(\zeta) - \sum\limits_{j=1}^{n}{z_{dj}(\zeta)\lambda_j'}+S_{dk}^{+} \leq  0,~~\zeta\in \mathcal{U}_{\Gamma_{d}}~~\forall~ d,& \\
	& \nonumber z_{dk}(\zeta)b_{dk}-S_{dk}^{+}  \leq 0,~~~\forall~ d,~~\zeta\in \mathcal{U}_{\Gamma_{d}} ~~\forall~ d,& \\
	& \nonumber x_{ik}(\zeta)a_{ik}-S_{ik}^{-}  \leq  0,~~~\forall ~i,~~\zeta\in \mathcal{U}_{\Gamma_{i}}~~\forall~ i,& \\
	\end{eqnarray*}
	where,  $\mathcal{U}_{\Gamma_{i}}=\{ \zeta \in \mathcal{Z}~|~~~ \|\zeta\|_{1} \le \Gamma_{i}  \},~\forall~ i$ and  $\mathcal{U}_{\Gamma_{d}}=\{ \zeta \in \mathcal{Z}~|~~~ \|\zeta\|_{1} \le \Gamma_{d}  \},~\forall~ d$ sets are known as budget uncertainty sets and $\Gamma_{i}, \Gamma_{d}$ is denotes the budget parameter. If the value of these budget parameters is 0, then again SBM model converts to a crisp model and if its value equals the total number of uncertain input and output data for individual constraint, then the robust SBM  model gives the best feasible robust solutions in the worst-case scenarios. Next, we derive the theorem from computing the tractable formulation of both stages robust SBM model under the budget uncertainty set.
	
\begin{thm}
		Computationally tractable formulation of the first stage robust SBM model \eqref{mod1.13} under the Budget uncertainty sets $\mathcal{U}_{\Gamma_{i}}=\left\{\zeta~|~~~\|\zeta\|_{1} \le \Gamma_{i} ,~\forall ~i\right\}$,  and
	$\mathcal{U}_{\Gamma_{d}}=$ $\left\{\zeta~|~\|\zeta\|_{1} \le \Gamma_{d} ,~\forall ~d\right\}$,  is equivalent to the following linear optimization problem:
	\begin{flalign}
	& \max ~~~w_k^1 & \nonumber\\
&	\hbox{subject to  } \Bigg\{ w_k^1-p_1-{\dfrac{1}{D}\left({{\sum\limits_{d=1}^{D}{b_{dk}}}}\right)}  \leq  0, \,\,\,p_1-\dfrac{1}{m}{\sum\limits_{i=1}^{m}{{a_{ik}}}}  \leq  1,\,\,\,-p_1 x_{ik}^{0}+\sum\limits_{j=1}^{n}x_{ij}^{0}\lambda_j'+S_{ik}^{-}+ \Gamma_{i}\kappa_{i}^{x}+\sum\limits_{\ell=1}^{L} \mu_{i\ell}^{x} \le 0, ~\forall~i,& \nonumber\\
	& \kappa_{i}^{x}+\mu_{i\ell}^{x} \ge -p_1 x_{ik}^{\ell}+\sum\limits_{j=1}^{n}x_{ij}^{\ell}\lambda_j',~\forall~i,\ell,\,\,\, p_1 z_{dk}^{0} - \sum_{j=1}^{n}z_{dj}^{0}\lambda_j'+S_{dk}+ \Gamma_{d}^{z}\kappa_{d}^{z}+\sum\limits_{\ell=1}^{L} \mu_{d\ell}^{z}   \leq  0,~~\forall~ d,& \nonumber\\
	& \kappa_{d}^{z}+\mu_{d\ell}^{z} \ge p_1 z_{dk}^{\ell} - \sum_{j=1}^{n}z_{dj}^{\ell}\lambda_j',~\forall~d,\ell,\,\,\, z_{dk}^{0}b_{dk}+\kappa^{z}+\sum_{\ell=1}^{L}\mu_{d\ell}-S_{dk}  \leq  0,~~~\forall~ d,\,\,\, \kappa^{z}+ \mu_{d\ell} \ge z_{dk}^{\ell}b_{dk},~\forall~d, \ell, & \nonumber\\
	& x_{ik}^{0}a_{ik}+\kappa^{x}+\sum_{\ell=1}^{L}\mu_{i\ell}-S_{ik}^{-} \leq  0,~~~\forall ~i, \,\,\, \kappa^{x}+\mu_{i\ell} \ge x_{ik}^{\ell}a_{ik},~\forall~i, \ell, \,\,\, p_1-\sum\limits_{j=1}^{n}{\lambda_j'} =  0,& \nonumber\\
	& \kappa^{x}\ge 0, \kappa^{z}\ge 0, \mu^{x}\ge 0, \mu^{z}\ge 0, \kappa\ge 0, \mu \ge 0, \,\,\, p_1 >0;\,\,\lambda_j'\geq 0 \,\forall j, \,\,\,{S_{ik}^-}\geq 0\,\forall i,\,\,\,{S_{dk}} \geq  0 \,\forall d, \Bigg\}&\label{mod1.015}
	\end{flalign}
	
	\end{thm}
	
	\begin{proof}
	    By using the strong duality approach, one can obtain the desired result.
	\end{proof}
	\begin{definition}
		Let ${E^b}_k^{1}=1/w_k^1$ be the first stage efficiency in budget perturbation in Model (\ref{mod1.015}). Let ${E^b}_k^{1*}$ be the optimal value of ${E^b}_k^{1}$. Then $DMU_{k}$ is called first-stage efficient in budget perturbation if ${E^b}_k^{1*}  = 1$; otherwise, $DMU_{k}$ is called first-stage inefficient in budget perturbation.
	\end{definition}
	
		\begin{thm}\label{Thm.effi7}
	If $(p_1^{*},~ \lambda_j^{*'},~S_{ik}^{*-},~S_{dk}^{*},\kappa^{*x}, \kappa^{z*}, \mu^{x*}, \mu^{z*}, \kappa^{*}, \mu^{*} )$ is the optimal solution of the model (\ref{mod1.015}), then $w_{k}^{1} \ge 1$ and consequently ${E^b}_k^{1} \in (0,~1]$.
	\end{thm}
	\begin{proof}
	    The proof of this theorem is also same as Theorem-\ref{Thm.effi1}.
	\end{proof}
	
	Further, the following theorem presents a tractable formulation of the second stage robust SBM model under budget uncertainty as:

		\begin{thm}
		Computationally tractable formulation of the second stage robust SBM model \eqref{09} under the Budget uncertainty sets $\mathcal{U}_{\Gamma_{r_1}}=\left\{\zeta~|~~~\|\zeta\|_{1} \le \Gamma_{r_{1}} ,~\forall ~r_1\right\}$, $\mathcal{U}_{\Gamma_{r_2}}=\left\{\zeta~|~~~\|\zeta\|_{1} \le \Gamma_{r_{2}} ,~\forall ~r_2\right\}$ and
	$\mathcal{U}_{\Gamma_{d}}=$ $\left\{\zeta~|~\|\zeta\|_{1} \le \Gamma_{d} ,~\forall ~d\right\}$,  is equivalent to the following linear optimization problem:
		\begin{flalign}
			&\max w_k^2& \nonumber\\
			&\mbox{subject to  } \Bigg\{w_k^{2}-p_2-\dfrac{1}{(s_1+s_2)}{\left(\sum\limits_{r_1=1}^{s_1}{{c_{r_1k}}}-\sum\limits_{r_2=1}^{s_2}{{c_{r_2k}}}\right)}\leq 0,\,\,p_2-{\dfrac{1}{D}\left({{\sum\limits_{d=1}^{D}{{b_{dk}}}}}\right)}\leq 1,\nonumber&\\
			& -p_2 z_{dk}^{0}+\sum\limits_{j=1}^{n}z_{dj}^{0}\lambda_j'+S_{dk}+ \Gamma_{d}^{z}\vartheta_{d}^{z}+\sum\limits_{\ell=1}^{L} \eta_{d\ell}^{z} \leq  0,~~\forall~ d,  \,\,\, \vartheta_{d}^{z}+\eta_{d\ell}^{z} \ge -p_2 z_{dk}^{\ell}+ \sum\limits_{j=1}^{n} z_{dj}^{\ell}\lambda_j', ~~\forall~ d, \ell,  \nonumber &\\
				& p_2 y_{r_1k}^{0} -\sum\limits_{j=1}^{n}y_{r_1j}^{0}\lambda_j'+{{S_{r_1k}^+}}+ \Gamma_{r_{1}}^{y}\vartheta_{r_{1}}^{y}+\sum\limits_{\ell=1}^{L} \eta_{r_{1}\ell}^{y} \leq 0,\,\,\,\,\forall r_1,\,\,\, \vartheta_{r_{1}}^{y}+ \eta_{r_{1}\ell}^{y} \ge  p_2 y_{r_1k}^{\ell}-\sum\limits_{j=1}^{n}y_{r_1j}^{\ell}\lambda_j',~~\forall~ r_{1}, \ell, &\nonumber\\
			& -p_2 y_{r_2k}^{0} +\sum\limits_{j=1}^{n}y_{r_2j}^{0}\lambda_j'+{{S_{r_2k}^+}}+\Gamma_{r_{2}}^{y}\vartheta_{r_{2}}^{y}+\sum\limits_{\ell=1}^{L} \eta_{r_{2}\ell}^{y} \leq 0,\,\,\,\,\forall r_2,\,\,\, \vartheta_{r_{2}}^{y}+ \eta_{r_{2}\ell}^{y} \ge -p_2 y_{r_2k}^{\ell} + \sum\limits_{j=1}^{n}y_{r_2j}^{\ell}\lambda_j',~~\forall~ r_{2}, \ell, \nonumber &
			\end{flalign}
			\begin{flalign}
			& p_2-\sum\limits_{j=1}^{n}{\lambda_j'}= 0,\,\, y_{r_1k}^{0}{c_{r_1k}}+\Gamma_{r_{1}}\vartheta_{r_{1}}^{1}+ \sum\limits_{\ell=1}^{L}\eta_{r_{1}\ell}^{1}-{S_{r_1k}^+}  \leq 0,\,\,\,\,\forall r_1,\,\,\,  \vartheta_{r_{1}}^{1}+ \eta_{r_{1}\ell}^{1} \ge  y_{r_1k}^{\ell}c_{r_1k} , \,\,\,\,\forall r_1, \ell,\nonumber &\\	
			&   y_{r_2k}^{0}{c_{r_2k}}+\Gamma_{r_{2}}\vartheta_{r_{1}}^{2}+ \sum\limits_{\ell=1}^{L}\eta_{r_{2}\ell}^{2}-{S_{r_2k}^+} \leq 0,\,\,\,\,\forall r_2,\,\,\,\vartheta_{r_{1}}^{2}+ \eta_{r_{2}\ell}^{2} \ge y_{r_2k}^{\ell}c_{r_2k}= 0, \,\,\,\,\forall r_2, \ell, \nonumber& \\
			&z_{dk}^{0}{b_{dk}}+\Gamma_{d}\vartheta_{d}^{3}+ \sum\limits_{\ell=1}^{L}\eta_{d\ell}^{3}-{S_{dk}} \leq 0, \,\,\,\,\forall d,\,\,\,\vartheta_{d}^{3}+ \sum\limits_{\ell=1}^{L}\eta_{d\ell}^{3} \ge z_{dk}^{\ell}b_{dk} =0 \,\,\,\,\forall d, \ell, \nonumber & \\
				&  \vartheta_{r_{1}}^{y}, \vartheta_{r_{2}}^{y},  \vartheta_{d}^{z}, \eta_{r_{1}\ell}^{y}, \eta_{r_{2}\ell}^{y}, \eta_{d\ell}^{z} \ge 0\,\,\,\,\forall r_{1}, r_2, d, \ell,\,\,\, \vartheta_{r_{1}}^{1}, \vartheta_{r_{2}}^{2},  \vartheta_{d}^{3}, \eta_{r_{1}\ell}^{1}, \eta_{r_{2}\ell}^{2}, \eta_{d\ell}^{3} \ge 0\,\,\,\,\forall r_{1}, r_2, d, \ell,\nonumber &\\
			&p_2 >0;\,\,\lambda_j'= 0, \,\forall j, \,\,\,{S_{dk}}\geq 0,\,\,\,{S_{r_1k}^+}\geq 0,\,\,\,{S_{r_2k}^+}\Bigg\} &\label{mod1.0015}
		\end{flalign}
	\end{thm}
	\begin{proof}
	    One can prove the desired result using the strong duality theory.
	\end{proof}
	\begin{definition}
		Let ${E^b}_k^{2}=1/w_k^2$ be the second-stage efficiency in budget perturbation in Model (\ref{mod1.0015}). Let ${E^b}_k^{2*}$ be the optimal value of ${E^b}_k^{2}$. Then $DMU_{k}$ is called second-stage efficient in budget perturbation if ${E^b}_k^{2*}  = 1$; otherwise, $DMU_{k}$ is called second-stage inefficient in budget perturbation.
	\end{definition}
\begin{definition}
The overall efficiency, ${E^b}_k^{O}$, of $DMU_k$ is determined using first and second stage efficiencies ${E^b}_k^{O*}={E^b}_k^{1*}\times {E^b}_k^{2*}$. $DMU_{k}$ is called efficient if ${E^b}_k^{O*}  = 1$, i.e. iff ${E^b}_k^{2*}=1$ and ${E^b}_k^{2*}=1$; otherwise, $DMU_{k}$ is called inefficient in budget perturbation.
	\end{definition}
	
	\begin{thm}\label{Thm.effi8}
	If $(p_2^{*},~ \lambda_j^{*'},{S_{r_1k}^{*+}},~S_{r_2k}^{*+},~S_{dk}^{*},\vartheta_{r_{1}}^{1*}, \vartheta_{r_{2}}^{2*},  \vartheta_{d}^{3*}, \eta_{r_{1}\ell}^{1*}, \eta_{r_{2}\ell}^{2*}, \eta_{d\ell}^{3*}, \vartheta_{r_{1}}^{*y}, \vartheta_{r_{2}}^{*y},  \vartheta_{d}^{*z}, \eta_{r_{1}\ell}^{*y}, \eta_{r_{2}\ell}^{*y}, \eta_{d\ell}^{*z})$ is the optimal solution of the model (\ref{mod1.0015}), then $w_{k}^{2} \ge 1$ and consequently ${E^b}_k^{2} \in (0,~1]$.
	\end{thm}
	\begin{proof}
	    The proof of this theorem is similar to Theorem-\ref{Thm.effi1}.
	\end{proof}

Further, the budget parameter regulates the number of variables that can take their maximum fluctuations simultaneously, allowing for a balance between optimality and robustness. Although in the above SBM formulations, the exact value of $\Gamma$ is applied to all DMUs, the decision-maker may assign different values of $\Gamma$ to the DMUs, reflecting their knowledge of the DMUs and attitude toward risks.	\\

	\section{Case Study: Indian Banks}
In modern India, the first bank was originated in the $18^{th}$ century. In 1770, the first bank of India was created and was named Bank of Hindustan. In 1786, the General Bank of India was originated, and this was failed in the year 1791. In 1806, the Bank of Calcutta was originated in Kolkata, which is in the West of Bengal in India, and this bank is renamed as Bank of Bengal in the year 1809. In 1840, the Bank of Bombay was originated in Bombay, Maharastra. In 1843, the Bank of Madras was originated in Chennai. The Indian presidency government founded the bank of Bengal, Bank of Bombay, and Bank of Maharastra. These presidency banks were working as quasi-central Banks. In 1921, the Imperial Bank of India was originated, and it was the merger of these three banks. In 1935, the Reserve Bank of India was created under the Act (1934) of India's Reserve Bank (R.B.I.). After India's Independence, the Imperial Bank of India was renamed as State Bank of India (S.B.I.) in the year 1955. S.B controlled India's eight banks. in 1960 under the Act (1959) of S.B.I. Now, these banks are called associate banks of S.B.I. The primary regulator and strategy developer of Indian banking are R.B.I., and the currency of India is Rupee. The supply, issue, and credit of rupee are controlled by R.B.I., the central bank of India. R.B.I carries the exercises supervision and monetary policy of India's banks and non-banking finance companies. Banks in India are classified into two types of banks scheduled and non-scheduled. The scheduled banks are organized in S.B.I. and its associates, nationalized, private sector, foreign banks, and regional rural banks.
	Now, Indian Banks are classified into commercial, payment, small finance, and co-operative banks.
	\begin{itemize}
		\item { Commercial Banks: } Commercial banks provide the business for profit of banks like offering investment products, making business loans, and accepting deposits. Commercial banks are considered in both scheduled and non-scheduled banks under the Act (1949) of Banking regulation. Commercial banks are classified in the Public sector, Regional rural, Private sector, Foreign banks.
		\begin{itemize}
			\item {Public Sector Banks:} The major bank in India is Public Sector Banks. The Indian government has a stake larger than 50\%, i.e., majority stake is held by the Indian government R.B.I. The Public sector banks have 22 Indian banks.
			\item {Private Sector Banks:} The Private sector banks in India are operated by private companies or private individuals, i.e., private individuals hold the majority of the share capital bank. The private sector has full shareholders. The private sector banks have 20 Indian banks.
		\end{itemize}
		\item {Payment Banks: } It accepts the deposit with the limit of 1,00,000 Rs per customer. It deals with saving accounts and current accounts. Also, it deals the debit cards, mobile and internet banking. The first payment bank of India is Bharti Airtel. It does not deal the credit cards and loans.
		\item {Small Finance Banks: } It deals the account lending, deposits, small business units such as unorganized sector units, small and micro industries.
		\item {Cooperative Banks: } It deal the retail banking, credit unions, building societies, saving banks.
	\end{itemize}

Banks play a critical part in the development of every country's economy. India has banks as well. The safety of public wealth, the propellant of the economy, the availability of low-cost loans, worldwide reach, rural development, and large-scale economies are all significant benefits of banks. This section uses data from 42 banks in India to illustrate the applicability of the robust DEA models. The banks examined are a set of Indian banks with branches operating all over India. In banking studies, there are two main approaches to evaluating input and output measures: both an intermediary and a production approach \cite{berger1997efficiency}. Banks are seen as intermediaries between savers and investors, primarily transferring capital and employee (inputs) to loans and securities (outputs) by the former \cite{fethi2010assessing}. In terms of the right way to use in literature, there is no universal consensus.  \cite{mostafa2009modeling} surveyed 26 research papers on the banking industry from various countries in order to pick the most suitable bank features. On the percentage of regular selection of these banking steps proposed in \cite{mostafa2009modeling}, employees are typically viewed as input (fixed input). However, the deposit is handled differently in banking studies;. At the same time, 15.38 percent of research papers used it as input to measure deposit, 26.92 percent of the surveyed papers used it as output to measure deposit \cite{lu2015robust}. As a consequence, the raw data are scaled for homogeneity and to minimize round-off errors in DEA models caused by extremely large values \cite{thanassoulis2001introduction}. Uncertainties compelling volatility in banks particular variables were considered in order to determine the efficiency and complexity of the robust. Uncertainties in the banking sector can arise from forecast things like loan and deposit, missing values, and calculation errors, among other things. If any of the input or output data of a DMU is uncertain, it is defined as uncertain DMU.



As a consequence, the raw data are scaled for homogeneity and to minimize round-off errors in DEA models caused by huge values \cite{thanassoulis2001introduction}. Uncertainties compelling volatilities in banks particular variables were considered in order to determine the efficiency and complexity of the robust DEA model in comparison to the DEA model. Uncertainties in the banking sector can arise from forecast things like loan and deposit, missing values, and calculation errors, among other things. If any of the inputs or outputs data of a DMU is uncertain, it is defined as uncertain DMU.

Banks are divided in two stages, first stage is pre-acquisition stage, in this stage inputs are number of employees ($x_{1}$), interest expenses ($x_{2}$), operating expenses ($x_{3}$) and Assets ($x_{4}$) and output is Deposits ($z_{1}$); and second is profit stage, in this input is Deposits ($z_{1}$) and outputs are Total income ($y_{1}$), Loans ($y_{2}$), Debts ($y_{3}$), profit/Loss ($y_{4}$), and NPAs (undesirable) ($y_{5}$). \\

	
	Mathematical terminology which is used in this paper for Indian Bank data is given as follows:
	
	\begin{itemize}
	    \item 	Employees: The number of working employees in each public and private sector bank.
	
	    \item Interest expenses: It includes interest money borrowed, R.B.I./inter-bank borrowings, operation expenses, and miscellaneous costs.
	    \item Operating expenses: It takes the R.B.I./inter-bank borrowings, operating expenses.
	    \item 	Assets: It considers cash on hand, R.B.I. balances, balances with Indian banks, money on call and short notice, balances with foreign banks, investments, and other assets.
	
	    \item      Deposits: It accepts demand deposits from banks and other financial institutions and saves bank deposits and term deposits from banks and other financial institutions.
	
         \item Total income: It includes interest/discounts on advances/bills, investment income, interest on R.B.I. and other inter-bank funds, and interest gained on various gains.

         \item 	Loans: It takes loans subjected to restructuring (standard assets, sub-standard assets, doubtful assets during the year).

		\item Debts: It takes the corporate debt restructured (standard, sub-standard, doubtful assets during the year).
		
	    \item Profit/Loss: It takes the net profit or net loss during the year.
	
	    \item Non-performing assets (NPAs): It takes addition, reduction, and write-off during the year.
	
	\end{itemize}
	\begin{figure}[!htb]
		\centering
		\includegraphics[width=0.65\linewidth]{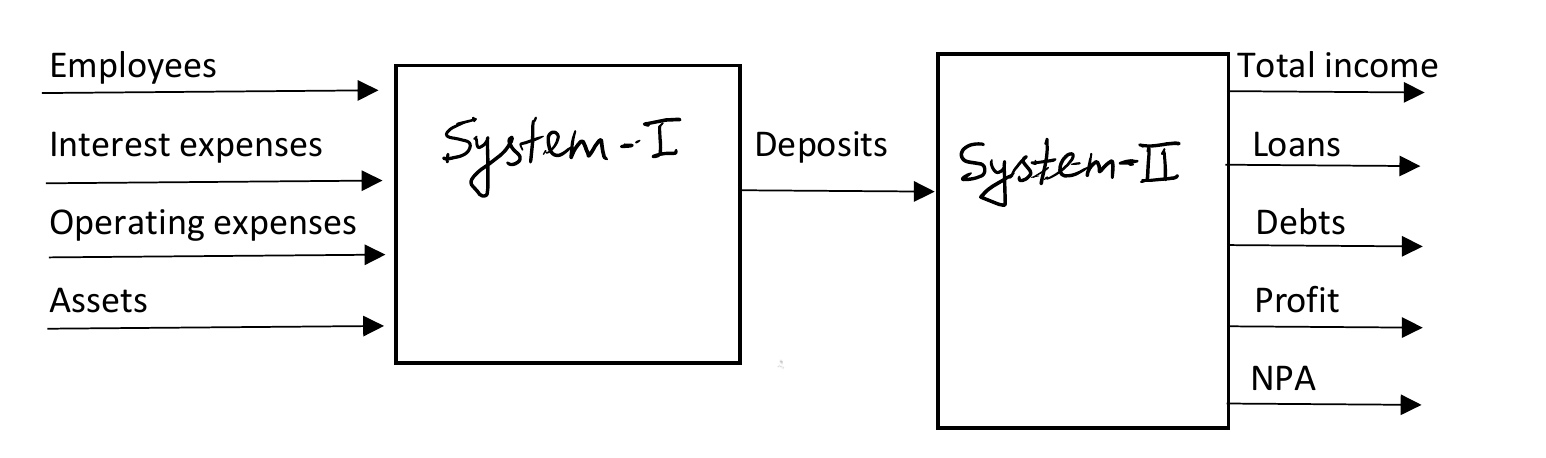}
		\caption{Graphical representation of Banks as two-stage}
		\label{fig:appnetwork2stage}
	\end{figure}
The descriptive statistics of 42 Indian banks' analysis of data collected is represented in Table \ref{tab:appstat}.

		\begin{table}[!htb]
	\centering
	\caption{\textbf{Descriptive Stat. for 42 Indian Banks}}\label{tab:appstat}
	\resizebox{\textwidth}{!}{
		\begin{tabular}{llllllllll}
			\hline
	Data variables & Mean & Standard Error & Median & Standard Deviation & Kurtosis & Skewness & Minimum & Maximum & Sum \\\hline
	Employees $(x_1)$ & 30165.54762 & 6637.095357 & 18233 & 43013.294 & 21.68767066 & 4.191538662 & 844 & 264041 & 1266953 \\
	Interest expended $(x_2)$ & 149864.8095 & 35581.59925 & 100000.5 & 230595.1184 & 25.98282779 & 4.66774452 & 4114 & 1456456 & 6294322 \\
	Operating expenses $(x_3)$& 59716.88095 & 15071.13852 & 30562 & 97672.14076 & 23.42008693 & 4.446734878 & 1437 & 599434 & 2508109 \\
	Assets $(x_4)$ & 3412872.31 & 856951.5095 & 2188412 & 5553680.524 & 24.82204601 & 4.562995354 & 81147 & 34547520 & 143340637 \\
	Deposits $(z_1)$ & 2684764.262 & 666553.997 & 1695684 & 4319763.616 & 25.54323925 & 4.634198403 & 73319 & 27063433 & 112760099 \\
	Total income $(y_1)$ & 273141.1429 & 66299.23182 & 170961 & 429668.13 & 23.52052708 & 4.433238318 & 6528 & 2651000 & 11471928 \\
	Loans $(y_2)$ & 51909.80952 & 12048.28308 & 21218.5 & 78081.7985 & 11.45468592 & 2.995793568 & 2 & 417313 & 2180212 \\
	Debts $(y_3)$ & 21909.42857 & 4838.774529 & 6809 & 31358.84302 & 11.45624445 & 2.787250858 & 214 & 170779 & 920196\\
	Profit/Loss $(y_4)$ & 81219.09524 & 10281.98666 & 100940 & 66634.88941 & -1.725655474 & -0.00185079 & 489 & 174867 & 3411202 \\
	NPAs $(y_5)$ & 123498.881 & 29569.84988 & 54904 & 191634.5296 & 16.97668662 & 3.620016142 & 383 & 1108547 & 5186953 \\\hline
	\end{tabular}}
		\small	Data is collected by Centre for Monitoring Indian Economy Pvt. Ltd. (CMIE).
\end{table}


	In general,  all these input and output parameters describe above are not explicitly fixed. The reason behind data uncertainty has already been explored in the introduction. Consequently, in these cases, there's always the potential that the data would deviate. In this case study, we consider two distinct sorts of potential deviations: $10\%$ of the nominal data and $20\%$ of the nominal data. Because there are two types of possible deviations, $L=2$. Further, we assume that information about deviation vectors is contained in several forms of uncertainty sets, such as ellipsoidal, polyhedral, and budget uncertainty sets. For the fixed $k$, uncertainty set is defined as:
	\begin{eqnarray}\label{mod2.0015}
	\mathcal{U}_{x} =\left\{ (x_{ik}, x_{ij}) ~|~ x_{ik}^{0}+\sum_{\ell=1}^{2}\zeta_{\ell}x_{ik}^{\ell}, ~ x_{ij}^{0}+\sum_{\ell=1}^{2}\zeta_{\ell}x_{ij}^{\ell}, ~~\zeta \in \mathcal{Z}_{i} ~\forall~i\right\}\\\label{mod2.0016}
	\mathcal{U}_{d} =\left\{ (z_{dk}, z_{dj}) ~|~ z_{dk}^{0}+\sum_{\ell=1}^{2}\zeta_{\ell}z_{dk}^{\ell}, ~ z_{dj}^{0}+\sum_{\ell=1}^{2}\zeta_{\ell}z_{dj}^{\ell}, ~~\zeta \in \mathcal{Z}_{d}, ~\forall~d\right\}
	\end{eqnarray}
	where $\zeta $ is a perturbation or deviation vector, $\mathcal{Z}_{x} \subset \mathcal{U}_{x}\in \mathcal{R}^{2} \hbox{  and } \mathcal{Z}_{d} \subset \mathcal{U}_{d} \in \mathcal{R}^{2}$ are the prescribed uncertainty set.
	

\subsection{\bf Ellipsoidal uncertainty set}
In this portion, we consider that the perturbation vector of the input and output data of all the 42 Indian banks are given in the form of an ellipsoidal uncertainty set for both stages of the SBM model. Mathematically, the ellipsoidal set is defined as:

\begin{equation*}
    \mathcal{Z}_{i}=\left\{\zeta~|~\sqrt{\zeta_{1}^2+\zeta_{2}^2} \le \Omega_{i},~\forall ~i\right\} \hbox{ and } \mathcal{Z}_{d}=\left\{\zeta~|~\sqrt{\zeta_{1}^2+\zeta_{2}^2} \le \Omega_{d},~\forall ~d\right\}.
\end{equation*}

Now, taking unit ellipsoidal uncertainty set, i.e. ($\Omega_{i}=\Omega_{d}=1, ~\forall~i,d$) and using Lingo 11.0 software, we solved both models using nominal and deviated data in the first and second stages of tractable SBM models  \eqref{mod1.14} and \eqref{Secondstageellip}, and computed the efficiency shown in the Figure-\ref{figR1}. Table \ref{tab:baeff} shows the efficiencies for all banks employing ellipsoidal uncertainty in RO, and columns 3 and 7 representing first and second stage efficiencies, respectively. In the ellipsoidal instance, the overall efficiencies for all banks are found by multiplying the efficiencies of the first and second stages because the first and second stages are connected in series, as shown in column 11.
\begin{table}[!htb]
		\centering
		\caption{\textbf{Efficiencies of the Banks}}\label{tab:baeff}
		\resizebox{\textwidth}{!}{
			\begin{tabular}{ccccccccccccc}
				\hline
				DMUs &	\multicolumn{4}{c}{First stage efficiency} &	\multicolumn{4}{c}{Second stage efficiency} &	 \multicolumn{4}{c}{Overall efficiency}\\\hline
				& \multirow{2}{*}{Crisp/Nominal SBM model} &   & Robust SBM model  &  & \multirow{2}{*}{Crisp/Nominal SBM model}  &  & Robust SBM model &   & \multirow{2}{*}{Crisp/Nominal SBM model}  &  & Robust SBM model &   \\ \cline{3-5} \cline{7-9} \cline{11-13}
				&                &  Ellipsoidal & Polyhedral & Budget  &  & Ellipsoidal & Polyhedral & Budget &   & Ellipsoidal & Polyhedral & Budget  \\ \hline
				AB    & 0.915143 & 0.923345 & 0.953017 & 0.915143 & 0.683844 &	0.747429 &	1 &	1        & 0.625815 & 0.690136 & 0.953017 &	 0.915143 \\
				ANB   & 1        & 1        & 1        & 1        & 0.760777 &	0.80889	 &  1 &	0.813979 & 0.76077  & 0.80889  & 1        & 0.813979 \\
				BoB   & 1        & 1        & 1        & 1        & 1        &  1	     &  1 &	1        & 1        &  1	   & 1        &	 1        \\
				BoI   & 1        & 1        & 1        & 1        & 0.827858 &	0.862479 &	1 &	0.827858 & 0.827857 & 0.862479 & 1        &	 0.827857 \\
				BoM   & 1        & 1        & 1        & 1        & 0.875011 &	0.900148 &	1 &	0.875011 & 0.87501  & 0.900148 & 1        &	 0.875011\\
				CB    & 0.934988 & 0.93993  & 0.965296 & 0.934988 & 0.908061 &	0.926551 &	1 &	1        & 0.873528 & 0.870893 & 0.965296 & 0.934988\\
				CBI   & 1        & 1        & 1        & 1        & 0.748257 &	0.798888 &	1 &	0.760988 & 0.748257 & 0.798887 & 1        &	 0.760988 \\
				CoB   & 0.899957 & 0.916468 & 0.968766 & 0.899957 & 1        &  1        &  1 &	1        & 0.899957 & 0.916468 & 0.968766 &	 0.899957 \\
				DB    & 0.831376 & 0.850984 & 0.970836 & 0.831376 & 1        &  1        &  1 &	1        & 0.831376 & 0.850984 & 0.970836 &	 0.831376 \\
				IDBI  & 1        & 1        & 1        & 1        & 1        &  1        &  1 &	1        & 1        &  1       &  1       &	 1        \\
				IB    & 0.97089  & 0.97663  & 0.988567 & 0.970891 & 0.675827 &	0.741025 &	1 &	0.720826 & 0.656155 & 0.723707 & 0.988567 &	 0.699843 \\
				IOB   & 0.847362 & 0.866638 & 1        & 0.847362 & 0.955082 &	0.964044 &	1 &	1        & 0.809301 & 0.835477 & 1        &	 0.847362 \\
				OBC   & 1        & 1        & 1        & 1        & 1        &  1        &  1 &	1        & 1        &  1       &  1       &	 1         \\
				PASB  & 1        & 1        & 1        & 1        & 0.683394 &	0.74707	 &  1 &	0.705094 & 0.683393 & 0.747069 & 1        &	 0.705094 \\
				PNB   & 1        & 1        & 1        & 1        & 0.674843 &	0.740239 &	1 &	0.67991  & 0.674843 & 0.740239 & 1        &	 0.679910 \\
				SBI   & 1        & 1        & 1        & 1        & 1        &  1        &  1 &	1        & 1        &  1       &  1       &	 1        \\
				SB    & 0.87749  & 0.880397 & 0.936408 & 0.867749 & 0.725866 &	0.781    &	1 &	0.74129  & 0.629870 &  0.68759 & 0.936407 &	 0.643254 \\
				UCOB  & 1        & 1        & 1        & 1        & 0.804125 &	0.84352	 &  1 &	0.804125 & 0.804125 & 0.843520 &	1     & 0.804125 \\
				UBI   & 1        & 1        & 1        & 1        & 0.814866 &	0.852101 &	1 &	0.814866 & 0.814866 & 0.852101 &	1     & 0.814866\\
				UnBoI & 1        & 1        & 1        & 1        & 1        &  1        &  1 &	1        &  1       &  1      &  1       &	 1   \\
				VB    & 1        & 1        & 1        & 1        & 0.523083 &	0.619001 &	1 &	1        & 0.523083 & 0.619001 & 	1 &	1 \\
				AXIS  & 0.712414 & 0.755616 & 1        & 0.712714 & 0.963365 &	0.970733 &	1 &	1        & 0.686604 & 0.733501 & 	1 &	 0.712714\\
				BBL   & 0.60703  & 0.640842 & 1        & 0.607603 &  1       &  1        &  1 &	1        & 0.606608 & 0.640842 & 	1 &	 0.607603 \\
				CSBL  & 1        & 1        & 1        & 1        & 0.90052  &	0.920527 &	1 &	1        & 0.900519 & 0.920527 &	1	&  1 \\
				CUBL  & 0.794045 & 0.815753 & 0.905729 & 0.794045 & 0.658396 &	0.7271   &	1 &	0.909114 & 0.522796 &	0.593134 & 	 0.905729008	& 0.721878\\
				DCB   & 0.687234 & 0.720167 & 1        & 0.687234 & 0.739176 &	0.791633 &	1 & 1        & 0.507961 &	0.570108 & 	1	        & 0.687234 \\
				FBL   & 0.878429 & 0.896505 & 0.972642 & 0.878529 & 0.493838 &	0.595637 &	1 &	1        & 0.433850 &	0.533992 & 	 0.972641539	& 0.878529 \\
				HDFC  & 0.846798 & 0.883893 & 1        & 0.846798 & 1        &  1        &  1 &	1        & 0.846798 &	0.883893 & 	1	        & 0.846798 \\
				ICICI & 0.71417  & 0.753078 & 1        & 0.714177 & 1        &  1        &  1 &	1        & 0.714177 &	0.753078 & 	1	        & 0.714177 \\
				IDFC  & 0.514533 & 0.593277 & 1        & 0.514533 & 1        &  1        &  1 &	1        & 0.514533 &	0.593277 & 	1	        & 0.514533 \\
			INDUSIND  & 0.65028  & 0.696719 & 1        & 0.650281 & 1        &  1        &  1 &	1        & 0.650281 &	0.696719 & 	1	        & 0.650281 \\
				JAKBL & 1        & 1        & 1        & 1        & 0.601617 &	0.68174  &	1 &	0.707071 & 0.601617 &	0.681739 & 	1	        & 0.707070 \\
				KBL   & 0.870011 & 0.881709 & 0.958245 & 0.870011 & 0.512722 &	0.610724 &	1 &	0.890941 & 0.446074 &	0.538481 & 	 0.958245414	& 0.775128 \\
				KVBL  & 0.864648 & 1        & 0.956739 & 0.804649 & 0.583563 &	0.667318 &	1 &	1        & 0.469563 &	0.667317 & 	 0.956739127	& 0.804648 \\
				KMBL  & 0.701251 & 0.736096 & 1        & 0.701250 & 0.866349 &	0.893229 &	1 &	0.868665 & 0.607528 &	0.657502 & 	1	        & 0.609152 \\
				LVBL  & 0.788489 & 0.816203 & 0.934744 & 0.788489 & 0.881718 &	0.905508 &	1 &	0.91497  & 0.695224 &	0.739078 & 	 0.934744548	& 0.721444 \\
				NBL   & 1        & 1        & 1        & 1        & 1        &  1        &  1 &	1        & 1        & 1        &  1        &  1  \\
				RBL   & 0.734734 & 0.778543 & 1        & 0.734734 & 0.72431  &	0.779757 &	1 &	1        & 0.532175	& 0.607075 & 	1	        & 0.734734 \\
				SIBL  & 0.917022 & 0.925750 & 0.957092 & 0.917022 & 0.486904 &	0.590099 &	1 &	1        & 0.446502	& 0.546283 &	 0.957092582	& 0.917022 \\
				TMBL  & 0.834759 & 0.854635 & 0.940642 & 0.834759 & 0.642422 &	0.714338 &	1 &	1        & 0.536267	& 0.610499 &	 0.940642628	& 0.834759\\
				TDB   & 0.797561 & 0.819232 & 1        & 0.797561 & 1        &  1        &  1 &	1        & 0.797561	& 0.819232 &	1	        & 0.797561 \\
				YES   & 0.815035 & 0.852235 & 1        & 0.815035 & 0.903366 &	0.922801 &	1 &	1        & 0.736275	& 0.786444 &	1	        & 0.815035 \\\hline
		\end{tabular}}

		\end{table}

 \begin{figure}[!htb]
		    \centering
		    \includegraphics[width=17cm, height=6cm]{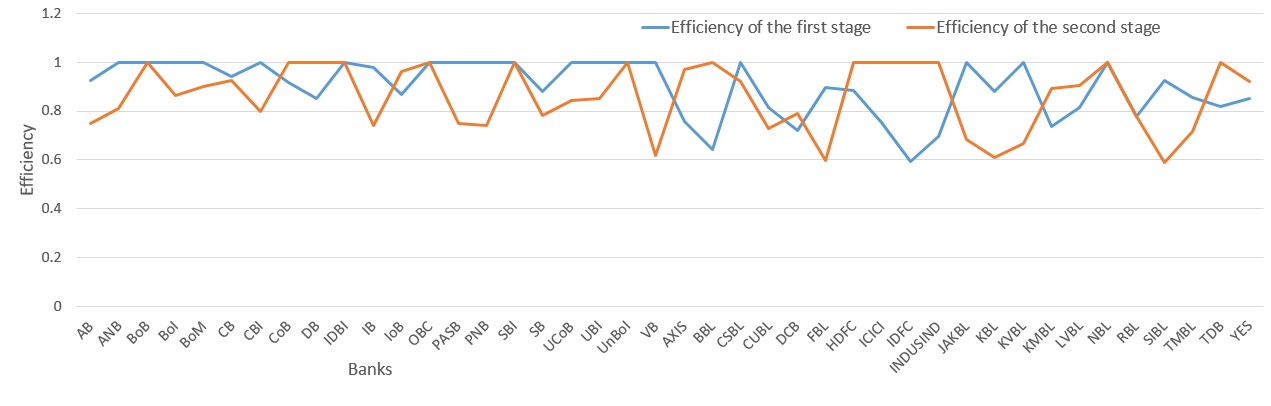}
		    \caption{Efficiencies for both stage SBM model under ellipsoidal uncertainty set}
		    \label{figR1}
		\end{figure}

The model's efficiency plays an essential part in the SBM framework. If the efficiency of any model in this framework is $1$, it implies that the model is efficient and has no chance of improvement. However, if the efficiency score is between $0$ to $1$, the model can improve and is considered praiseworthy under the DMUs. From the Figure-\ref{figR1} one may observe that a total of $18$ banks (ANB, BoB, BoI, BoM, CBI, IDBI, OBC, PASB, PNB, SBI, UCoB, UBI, UnBoI, VB, CSBL, JAKBL, KVBL, and NBL) are efficient in the first stage of the robust SBM model under the ellipsoidal uncertainty set. In contrast, total $14$ banks (BoB, CoB, DB, IDBI, OBC, SBI, UnBoI, BBL, HDFC, ICICI, IDFC, INDUS, NBL, TDB) are efficient in the second stage. Aside from that, the efficiency of the remaining models in both stages lies in $(0,1)$, with no one model being insufficient.

In the first stage robust SBM model, the efficiency of IDFC banks is $\textbf{0.593277}$, which is the lowest among all of them. Therefore, compared to other banks whose efficiency ranges from $0$ to $1$, the possibilities of improvement for this bank are pretty high. On the other hand, in the second stage efficiency of SIBL banks is $\textbf{0.590099}$, which is the lowest among all of them. Therefore, compared to other banks, the possibilities of improvement for this bank are quite high. Bank ANB is efficient in the first stage, in which the first stage's inputs are fully utilized to produce the first stage's outputs (intermediates), but inefficient in the second stage, in which the second stage's inputs (intermediates) are not fully utilized to produce the final outputs, resulting in overall inefficiency. As a result, the ANB is not effectively utilizing intermediates although this cannot be determined using black-box model.

\subsection{\bf Polyhedral uncertainty set}
Here, we consider that perturbation vector $\zeta_{1}$ and $\zeta_{2}$ of the input and output data are given in the form of a polyhedral uncertainty set for both stages robust SBM models. In this first stage robust SBM model, we consider only affine perturbation of the input and output parameters. Therefore, in this case all the matrices $H_{i}=H_{d}=\textbf{O}, ~\forall ~i,d$, where $\textbf{O} \in \mathcal{R}^{K \times L}$ denotes the zero matrix. As a result, uncertainty sets $\mathcal{Z}_{i}~\forall ~i,$ and $\mathcal{Z}_{d} ~\forall ~d$ will became empty sets.
However, in the second stage, the data regarding deposits and earning are affinely defined, thus corresponding the uncertainty sets $\mathcal{Z}_{r_{1}}~\forall ~r_{1},$ and $\mathcal{Z}_{d} ~\forall ~d$ become and empty and their nominal values are $y_{r_{1}j}^{0},~ \forall ~r_{1}, j$, $y_{dj}^{0},~ \forall ~d, j$ $y_{r_{2}j}^{0},~ \forall ~r_{2}, j$. Information about possible deviation has already been defined. In this stage, a few more restrictions have been placed in a few Banks. The names of those banks for $r_2=1$, i.e., constraint related Loans, are AB, ANB, PASB, VB, BBL, CSBL, CUBL, and DCB. Further, for $r_{2}=2$ i.e., constraint related debts, those banks names are ANB, PASB, VB, BBL, CSBL, SUBL, DCB, FBL, HDFC, INDUSIND, KVBL, NBL, RBL, SIBL, TDB, YES. Again, for $r_{2}=3$ i.e., constraint related profit and $r_{2}=4$ i.e., constraint related NPAs, banks names are AXIS, DCB, JAKBL, NBL and BBL, CSBL, DCB, NBL respectively. Thus, with the help of the equations \eqref{mod2.0015} and \eqref{mod2.0016}, the uncertainty set is defined as follows:
\begin{eqnarray*}
\mathcal{Z}_{r_{2}} = \left\{ \zeta~|~ H_{1}\zeta +q_{1}\ge 0,~ H_{2}\zeta +q_{2}\ge 0,~ H_{3}\zeta +q_{3}\ge 0,~ H_{4}\zeta +q_{4}\ge 0 \right\}
\end{eqnarray*}

where
\begin{eqnarray*}
H_{1}^{T} & = & \begin{bmatrix} 945 & 1649.7 & 636.7 & 49.7 & 49.7 & 55.8 & 84.8 & 68.3 \\
 1890 & 3299.4 & 1273.4 & 99.4 & 99.4 & 111.6 & 169.6 & 136.6 \\
\end{bmatrix} \\
H_{2}^{T} & = &
\scriptsize{\begin{bmatrix} 61.5 & 69.6 & 68.3 & 22.4 & 22.4 & 68.3 & 68.3 & 100.2 & 85.2 & 82.5 & 96.9 & 68.3 & 85.2 & 45 & 21.4 & 24.8\\
123 &  139.2 & 136.6 & 44.8 & 44.8 & 136.6 & 136.6 & 200.4 & 170.4 & 165 & 193.8 & 136.6 & 170.4 & 90 & 42.8 & 49.6 \\
\end{bmatrix}}\\
H_{3}^{T} & = & \begin{bmatrix} 275.7 &  245.3 & 202.7 & 48.9 \\
                                551.4 & 490.6 & 405.4 & 97.8
\end{bmatrix}\\
 H_{4}^{T}& = & \begin{bmatrix} 172.9 & 416.3 & 146.7 & 38.3 \\
                                345.8 & 832.6 & 293.4 & 76.6
\end{bmatrix}
\end{eqnarray*}

and
\begin{eqnarray*}
q_{1} & =& (1700,~ 503,~ 1633,~ 503,~ 703,~ 342,~ 152,~ 269) \\
q_{2} & =& (435,~ 431,~ 317,~ 326,~ 526,~ 317,~ 517,~ 489,~ 348,~ 198,~ 331,~ 217,~ 348,~ 150,~ 586,~ 452) \\
q_{3} & =& (443,~ 457,~ 473,~ 211) \\
q_{4} & =& (771,~ 837,~ 833,~ 417) \\
\end{eqnarray*}

where all the values of the matrix $H_{r_{2}}~\forall~ r_{2}$ are computed using $10\%$ and $20\%$ of the nominal value of those banks where extra conditions are being imposed, now, after putting all the information of the input and output data in both models \eqref{mod1.15} and \eqref{012}, we have computed the efficiencies of both stages model with the help of Lingo 11.0 software that has been shown in the Figure-\ref{figR2}. Table \ref{tab:baeff} shows the efficiencies for all banks employing polyhedral uncertainty in RO, and columns 4 and 8 representing first and second stage efficiencies, respectively. In the polyhedral instance, the overall efficiencies for all banks are presented in column 12.

\begin{figure}[!htb]
		    \centering
		    \includegraphics[width=17cm, height=6cm]{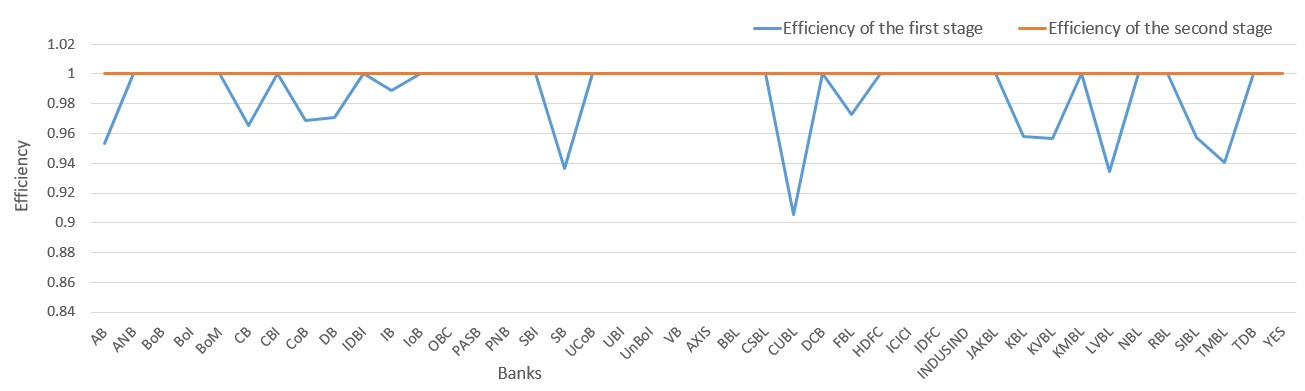}
		    \caption{Efficiencies for both stage SBM model under polyhedral uncertainty set}
		    \label{figR2}
		\end{figure}
	
From the Figure-\ref{figR2} one may observe that a total of $13$ banks (AB, CB, CoB, DB, IB, SB, CUBL, FBL, KBL, KVBL, LVBL, SIBL, TMBL) efficiency lies between $0$ to $1$ in which CUBL bank has the lowest efficiency (\textbf{0.905729}), whereas, in the second stage, all models of banks are efficient. Apart from that, no one model is insufficient. The proposed models give more information, for ex: Bank AB is inefficient in the first stage, in which the first stage's inputs are not fully utilized to produce the first stage's outputs (intermediates), but efficient in the second stage, in which the second stage's inputs (intermediates) are fully utilized to produce the final outputs, resulting in overall inefficiency. As a result, the AB does not effectively utilize initial inputs although this cannot be determined using the black-box model.

\subsection{\bf Budget uncertainty set}
Again, we consider that the perturbation vector of the input and output data is given in the form of a budget uncertainty set for both stages of the SBM models. Mathematically, the first stage of this set is defined as:
\begin{eqnarray*}
\mathcal{U}_{\Gamma_{i}}=\{\zeta~~|~~|\zeta_{1}|+|\zeta_{2}| \le \Gamma_{i}~~~\forall~ i=1,2,3,4\} \hbox{ and } \mathcal{U}_{\Gamma_{d}}=\{\zeta~~|~~|\zeta_{1}|+|\zeta_{2}| \le \Gamma_{d}~~~\forall~ d=1\}.
\end{eqnarray*}
Similarly, we can define the second stage uncertainty sets. The value of the budget parameters $\Gamma_i$ and $\Gamma_d$ can be select from the interval $[0, 42]$ since the first uncertain constraint of the first stage robust SBM model \eqref{mod1.015} includes 42 uncertain parameters. If all the values of budget parameter $\Gamma_{i}=\Gamma_{d}=0$, then the robust SBM model will transform into the crisp model. In contrast, if the value of  $\Gamma_{i}=\Gamma_{d}=42$, then the robust SBM model gives the feasible solution in the worst-case scenario. Similarly, we might choose alternative budget parameter values for the remaining uncertain constraints. The choice of value selection of budget parameters entirely depends on the intention of the decision-maker. In this case, tractable formulation of both stages robust SBM models is linear. Now, after putting all the information of the input and output data in both models \eqref{mod1.015} and \eqref{mod1.0015}, we have computed the efficiencies of both stages model with the help of Lingo 11.0 software that has been shown in the figure \ref{figR3}. Table \ref{tab:baeff} shows the efficiencies for all banks employing budget uncertainty in RO, and columns 5 and 9 representing first and second stage efficiencies, respectively. In the budget uncertainty set instance, the overall efficiencies for all banks are presented in column 13.
\begin{figure}[!htb]
		    \centering
		    \includegraphics[width=17cm, height=6cm]{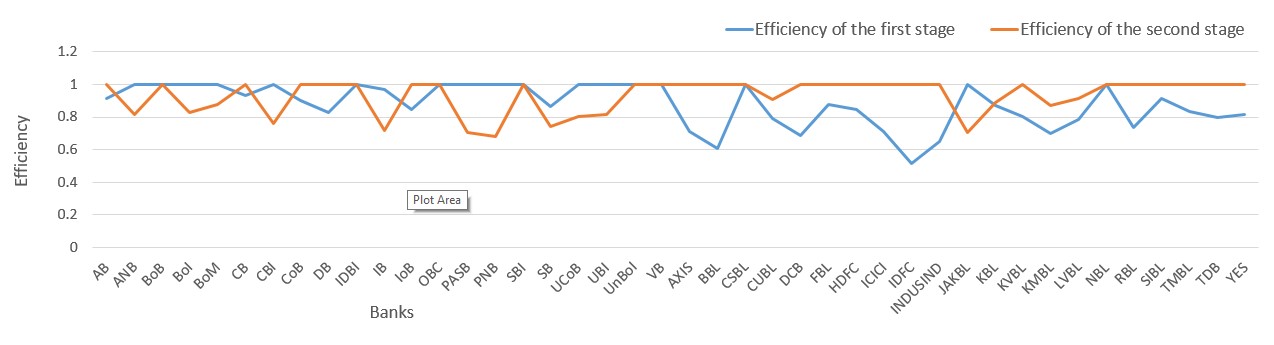}
		    \caption{Efficiencies for both stage SBM model under budget uncertainty set}
		    \label{figR3}
		\end{figure}

In Figure-\ref{figR3} one may observe that a total of $17$ banks (ANB, BoB, BoI, BoM, CBI, IDBI, OBC, PASB, PNB, SBI, UCoB, UBI, UnBoI, VB, CSBL, JAKBL, and NBL) are efficient in the first stage of the robust SBM model under the ellipsoidal uncertainty set, whereas total $27$ banks (AB, BoB, CB, CoB, DB
IDBI, IoB, OBC, SBI, UnBoI, VB, AXIS, BBL, CSBL, DCB, FBL, HDFC, ICICI, IDFC, INDUSIND, KVBL, NBL, RBL, SIBL, TMBL, TDB, and YES) are efficient in the second stage. Here, in the first stage, robust SBM model IDFC banks have the lowest efficiency ($\textbf{0.544534}$), whereas PNB banks have the lowest efficiency in the second stage robust SBM model ($\textbf{0.67991}$). As a result, in comparison to other banks with efficiency levels ranging from $0 \hbox{ to } 1$, the possibilities of improvement for these banks are quite high. The proposed models give more information, for ex: Bank BoI is efficient in the first stage, in which the first stage's inputs are fully utilized to produce the first stage's outputs (intermediates), but inefficient in the second stage, in which the second stage's inputs (intermediates) are not fully utilized to produce the final outputs, resulting in overall inefficiency. As a result, the AB is not effectively utilizing intermediates, although this cannot be determined using the black-box model.
\subsection{Results Analysis}
 A case study on Indian banks is discussed in the preceding section. This part evaluates and contrasts the efficiency of the crisp/nominal and uncertain models of both stages of the robust SBM model. In the robust SBM model, we obtained three different models' efficiencies. Various forms of uncertainty sets, such as ellipsoidal, polyhedral, and budget uncertainty sets, distinguish these models.



Figure-\ref{figR4} provides a comparison between efficiencies for crisp and robust SBM models in the first stage. 
		\begin{figure}[!htb]
		    \centering
		    \includegraphics[width=17cm, height=6cm]{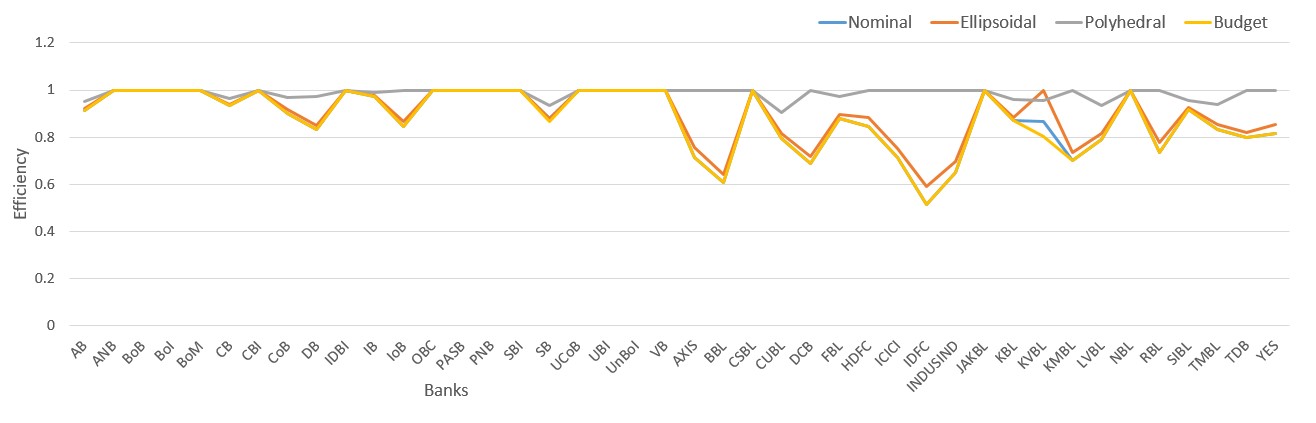}
		    \caption{Comparison of the first stage efficiencies for all the banks}
		    \label{figR4}
		\end{figure}

In this figure, it can be easily seen that no model provides an inefficient solution in any of the instances. In the robust SBM model, for the polyhedral case (i.e., when the data information is given a polyhedral uncertainty set), a maximum of 29 banks give efficient solutions. In contrast, in the case of nominal, ellipsoidal, and budget, a total of 17, 18, and 17 banks, respectively, provides an efficient solution. Furthermore, there are few banks (IoB, AXIS, BBL, DCB, HDFC, ICICI, IDFC, INDUSIND, KMBL, RBL, TDB, YES) whose efficiency lies in $(0,1)$ in case of nominal, ellipsoidal, and budget, whereas in the case of the polyhedral same banks become efficient. In the case of nominal, polyhedral, and budget,  KVBL bank efficiency lies in $(0,1)$, but it becomes efficient in the case of the ellipsoidal. The efficiency of the remaining banks is either lies in $(0,1)$ or the model becomes efficient. The lowest efficiencies of all the models in the first stage SBM model are presented in the following Table-\ref{mod.1tab1}:

\begin{table}[h!]
    \centering
    \begin{tabular}{ccc}
    \hline
       Model  &  Bank & Lowest efficiency \\\hline
       Nominal/Crisp  & IDFC  & 0.514533\\
       Ellipsoidal    & IDFC  & 0.593277\\
       Polyhedral     & CUBL  & 0.905729\\
       Budget         & IDFC  & 0.514534\\
       \hline
    \end{tabular}
    \caption{Lowest efficiencies in the first stage SBM model}
    \label{mod.1tab1}
\end{table}		
Therefore, in the first stage of the SBM model, the possibilities of improvement for these banks are quite high.

	Figure-\ref{figR5} compares efficiency in the second stage SBM model for all of the scenarios mentioned above.

\begin{figure}[!htb]
		    \centering
		    \includegraphics[width=15cm, height=6cm]{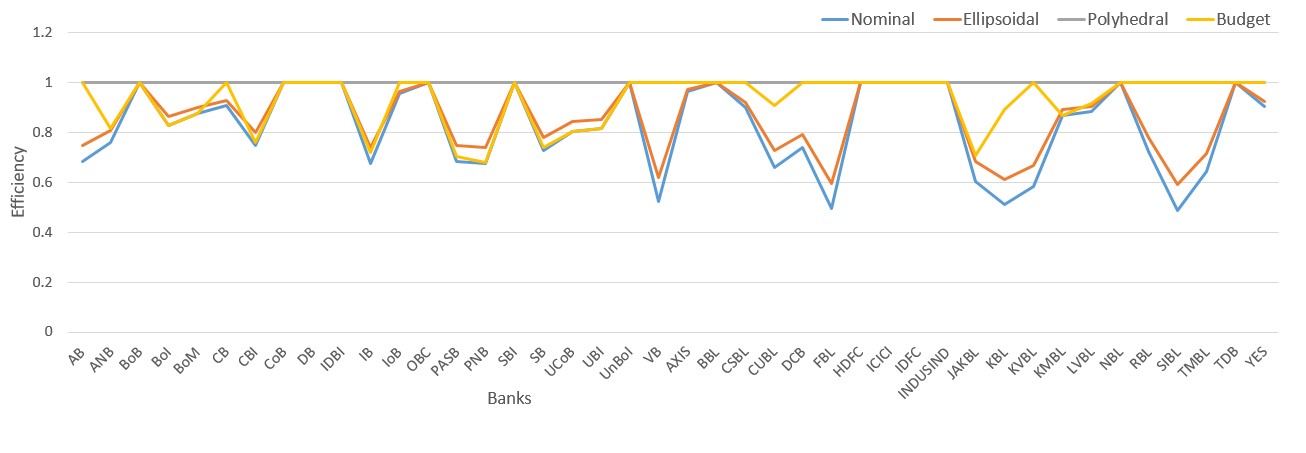}
		     \caption{Comparison of the second stage efficiencies for all the banks}
		     \label{figR5}
		\end{figure}
	In this figure, there isn't a single model that offers an inefficient solution. In the robust SBM model, for the polyhedral case, all the banks give efficient solutions, while in the case of nominal, ellipsoidal, and budget, a total of 14, 14, and 27 banks, respectively, provide the efficient solution. Further, there are also few banks (AB, CB, IoB, VB,  AXIS,  CSBL, DCB, FBL, KVBL, RBL, SIBL, TMBL, YES) whose efficiency lies in $(0,1)$ in the case of nominal and ellipsoidal, whereas in the case of the budget similar banks become efficient. The remaining banks' efficiency lies in $(0,1)$, or the banks become efficient. The lowest efficiencies of all the models, in this case, are presented in the following Table-\ref{mod.1tab2}:			
 \begin{table}[!htb]
    \centering
    \begin{tabular}{ccc}
    \hline
       Model  &  Bank & Lowest efficiency \\\hline
       Nominal/Crisp  & SIBL  & 0.486904\\
       Ellipsoidal    & SIBL  & 0.590099\\
       Budget         & PNB  & 0.67991\\
       \hline
    \end{tabular}
    \caption{Lowest efficiencies in the second stage SBM model}
     \label{mod.1tab2}
\end{table}

	In addition, Figure-\ref{figR6} compares the overall efficiency of crisp and robust SBM models.

	\begin{figure}[!htb]
		    \centering
		    \includegraphics[width=17cm, height=6cm]{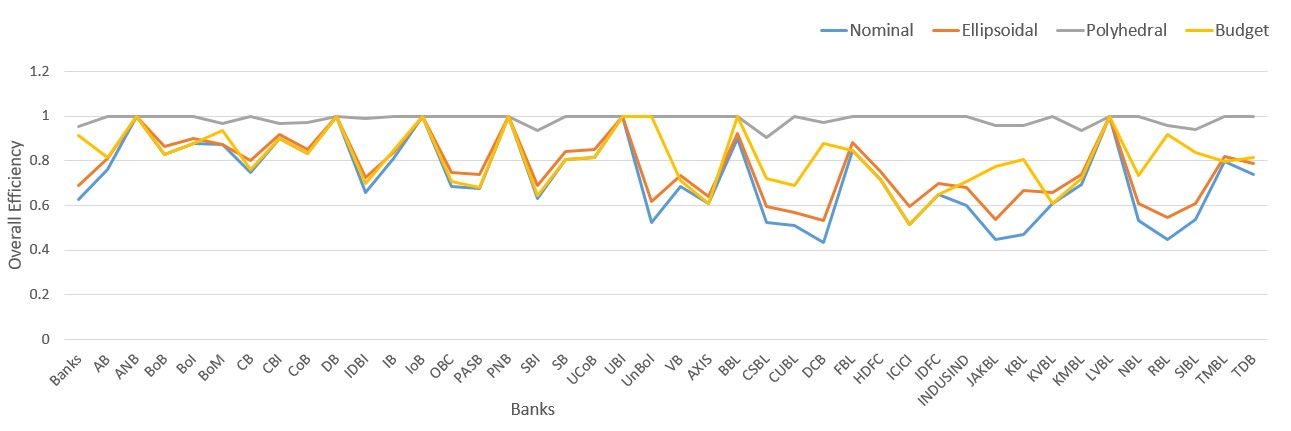}
		    \caption{Comparison of the overall efficiencies for all the banks}
		    \label{figR6}
		\end{figure}
		From this figure, it can be easily seen that no model provides an inefficient solution. In the robust SBM model, for the polyhedral case, a maximum of 29 models give efficient solutions. In contrast, in the case of nominal, ellipsoidal, and budget, a total of 6, 6, and 8 banks, respectively, provides an efficient solution. In addition, there are 21 banks (ANB, BoI, BoM, CBI, IoB, PASB, PNB, UCoB, UBI, AXIS, BBL, DCB,  HDFC, ICICI, IDFC, INDUSIND, JAKBL, KMBL, RBL, TDB, YES) whose efficiency lies between 0 to 1 in the case of nominal, ellipsoidal, and Budget, whereas in the case of the polyhedral similar banks become efficient. Further, in two banks, namely VB and CSBL, whose efficiency lies in $(0,1)$ in the case of nominal and ellipsoidal. But, in the case of the polyhedral and budget, similar banks become efficient. The remaining banks' efficiency either lies between the interval [0, 1] or the model becomes efficient. The lowest efficiencies of all the models, in this case, are presented in the following Table-\ref{mod.1tab3}:
\begin{table}[h!]
    \centering
    \begin{tabular}{ccc}
    \hline
       Model  &  Bank & Lowest efficiency \\\hline
       Nominal/Crisp  & DCB  & 0.433851\\
       Ellipsoidal    & DCB  & 0.533932\\
       Polyhedral    & CSBL  & 0.905729\\
       Budget         & ICICI  & 0.514534\\
       \hline
    \end{tabular}
    \caption{Lowest efficiencies in the second stage SBM model}
     \label{mod.1tab3}
\end{table}	

After reviewing all of the models' efficiencies, one can say that small perturbations in the data could heavily affect the feasibility and efficiencies of the models.

To evaluate the statistical significance of the results for all the models, we apply the  \cite{friedman1937use} rank test. Crisp, ellipsoidal, polyhedral, and budget efficiencies are used to calculate the Friedman test results. In terms of efficiency measures, the Friedman test has three degrees of freedom. The asymptotic p-value is 0.000, which is less than the significance level (0.05), allowing us to reject the null hypothesis (the performance of all models is equal). The performance of all the models can then be stated to be significantly different.
In a nutshell, Banks BoB, OBC, SBI, UnBoI, NBL are the most competitive and productive among
all banks. We would say that the proposed method's evaluation results are
more appropriate, accurate, and true to reality. According to the above discussion, some
robust two-stage SBM virtues are (i) it can deal with two-stage SBM with robust data; (ii) it can deal with robust two-stage SBM with missing, negative and undesirable data; (iii) it can deal with uncertain information in bank problems. So, a robust two-stage SBM is more appropriate to deal with some bank problems under uncertain situations.

\section{Conclusion}
The data in network SBM models are often used as precise, and this study attempted to alleviate network SBM models to uncertain situations. We developed the robust network SBM model in the presence of robust perturbation variables. The proposed robust two-stage SBM models assess the performance of the DMUs with negative, missing, and undesirable data. Since the data information is considered three distinct types of uncertainty sets, ellipsoidal, polyhedral, and budget uncertainty, the max-min approach is used to reformulate into the tractable forms of both stage SBM model. Further, we calculated the efficiencies for the first and second stages of each DMU. Then, overall efficiencies are calculated using first and second-stage efficiencies for ellipsoidal, polyhedral, and budget uncertainty perturbation. Finally, we applied the proposed methods to the Indian Banking sector to illustrate the applicability and efficacy of the developed method in the presence of robust data. The findings suggest that the efficiency values of the Indian banks range from 0.4338 to 1 implying that, the banking sector is generally inefficient. There would be a great perspective to optimize the production process further to improve the banking sector efficiency in India. The developed evaluation method can aid decision-makers substantially in enhancing the operational effectiveness and efficiency of the Indian banking sector. Based on robust optimization techniques, this paper suggested robust two-stage SBM models with a specific emphasis on assessing relative efficiency and improving decision-making for deciding effective combinations of computational operators.

  In future research, it is worthwhile to explore this study to a distributionally robust network SBM where the information of the data is given in an ambiguity set, i.e., a set of the family of distributions.

\section*{CRediT authorship contribution statement}
Both authors contributed equally to this research.

\bibliographystyle{unsrtnat}
\bibliography{references}  






\end{document}